\documentclass[DIV=10]{scrartcl}

\usepackage{comment}
\usepackage{subfig}
\usepackage{tabularx}
\usepackage{makecell}
\usepackage{amsfonts}
\usepackage{amsmath}
\usepackage{amsthm}
\usepackage[%
    hyperfootnotes=false,
    colorlinks = true,
    final=true,
    plainpages=false,
    pdfstartview=FitV,
    pdftoolbar=true,
    pdfmenubar=true,
    pdfencoding=auto, 
    psdextra, 
    bookmarksopen=true,
    bookmarksnumbered=true,
    breaklinks=true,
    linktocpage=true,
    citecolor=blue,
    linkcolor=blue]%
    {hyperref}
\usepackage[nameinlink]{cleveref}
\Crefname{figure}{}{}

\usepackage[style=numeric-comp,%
			defernumbers=true,
			useprefix=true,%
			giveninits=true,%
			hyperref=true,%
			uniquename=init,%
			sorting = none,
			sortcites=false,
			doi=true,%
			isbn=false,%
			url=false,%
			backend=biber%
	]{biblatex}
\usepackage{graphicx}
\usepackage{titling}
\thanksmarkseries{arabic}

\newtheorem{theorem}{Theorem}
\newtheorem{proposition}[theorem]{Proposition}%
\newtheorem{lemma}[theorem]{Lemma}

\newtheorem{remark}{Remark}%

\newtheorem{assumption}{Assumption}

\newcommand{\data}{f^\delta}
\newcommand{\dataNoisefree}{f}
\newcommand{\unknown}{u}
\newcommand{\operator}{A}

\newcommand{\noise}{\nu}

\newcommand{\reconstructionOperator}{R}
\newcommand*{\ft}{\ensuremath{\mathcal{F}}}
\newcommand*{\ift}{\ensuremath{\mathcal{F}^{-1}}}
\newcommand{\bR}{\ensuremath{\mathbb{R}}}

\newcommand{\net}{\mathcal{G}}

\newcommand*\diff{\mathop{}\!\mathrm{d}}
\newcommand{\review}[1]{{#1}}

\addtokomafont{disposition}{\rmfamily}

\makeatletter
\let\blx@rerun@biber\relax
\makeatother

\title{Convergent Data-driven Regularizations for CT Reconstruction}
\author{Samira Kabri\thanks{Helmholtz Imaging, Deutsches Elektronen-Synchrotron DESY, Notkestr. 85, Hamburg, 22607,
Germany}
\and 
Alexander Auras\thanks{Institute for Vision and Graphics, University of Siegen, Adolf-Reichwein-Straße 2a, Siegen, 57076, Germany}
\and
Danilo Riccio\thanks{School of Mathematical Sciences, Queen Mary University of London, Mile End Road, London, E1 4NS, United Kingdom}
\and
Hartmut Bauermeister\textsuperscript{2}%
\and
Martin Benning\textsuperscript{3,}\thanks{The Alan Turing Institute, British Library, 96 Euston Road, London, NW1 2DB, United Kingdom}%
\and
Michael Moeller\textsuperscript{2}%
\and
Martin Burger\textsuperscript{1,}\thanks{Fachbereich Mathematik, Universit\"at Hamburg, Bundesstrasse 55, Hamburg, 20146, Germany}
}
\date{Corresponding author:
\href{mailto:samira.kabri@desy.de}{samira.kabri@desy.de}}

\begin{document}
\maketitle%
\begin{abstract}
The reconstruction of images from their corresponding noisy Radon transform is a typical example of an ill-posed linear inverse problem as arising in the application of computerized tomography (CT). As the (na\"{\i}ve) solution does not depend on the measured data continuously, \textit{regularization} is needed to re-establish a continuous dependence. In this work, we investigate simple, but yet still provably convergent approaches to \textit{learning} linear regularization methods from data. More specifically, we analyze two approaches: One generic linear regularization that learns how to manipulate the singular values of the linear operator in an extension of \review{our previous work}, and one tailored approach in the Fourier domain that is specific to CT-reconstruction. We prove that such approaches become convergent regularization methods as well as the fact that the reconstructions they provide are typically much smoother than the training data they were trained on. Finally, we compare the spectral as well as the Fourier-based approaches for CT-reconstruction numerically, discuss their advantages and disadvantages and investigate the effect of discretization errors at different resolutions. 
\end{abstract}

\begin{minipage}[t]{.15\textwidth}%
\textbf{Keywords:}
\end{minipage}%
\begin{minipage}[t]{.75\textwidth}%
Inverse Problems, Regularization, Computerized Tomography, Machine Learning
\end{minipage}

\section{Introduction}\label{sec1}
Linear inverse problems are at the heart of a variety of imaging applications, including restoration tasks such as image deblurring as well as the inference of unknown images from measurements that contain implicit information about them as, for instance, arising in computerized tomography (CT), positron emission tomography (PET) or magnetic resonance imaging (MRI). All of these problems are commonly modeled as the task of recovering an image $\unknown$ from measurements
\begin{align}
    \label{eq:inverseProblem}
    \data = \operator \unknown + \noise
\end{align}
for a linear operator $\operator$ and noise $\noise$ characterized by some error bound (\textit{noise level}) $\delta \in [0,\infty[$ to be specified later. A key challenge for most practically relevant problems is that $\operator: X \rightarrow \review{Y}$ is a compact linear operator with infinite dimensional range (assuming here that $X$ and $Y$ are infinite-dimensional Hilbert spaces), leading to zero being an accumulation point of its singular values, and making the pseudo-inverse $\operator^\dagger$ discontinuous. 

\noindent \review{The compactness of $A$ makes it possible to expand it by its singular values, which means to write it in the form
\begin{align}
\label{eq:svd}
    \operator \unknown = \sum_{n=1}^\infty \sigma_n \langle \unknown, u_n \rangle v_n.
\end{align}
where the singular values $\sigma_n > 0$ are non-increasing, $\{u_n\}_{n \in \mathbb{N}}$ and $\{v_n\}_{n \in \mathbb{N}}$ form orthonormal bases of the orthogonal complement of the nullspace of the operator $\mathcal{N}(A)^\perp$, or the closure of its range $\overline{\mathcal{R}(A)}$, respectively. With a slight abuse of notation, $\langle \cdot , \cdot \rangle$ denotes the inner-product on the spaces $X$ and $Y$.}
Classical linear regularization strategies therefore aim to approximate $\unknown$ with so-called \emph{spectral regularization operators} of the form 
\begin{align}
\label{eq:reconstructionSVD}
    \reconstructionOperator(\data;g_\delta) =  \sum_{n=1}^\infty g_\delta(\sigma_n) \langle \data, v_n \rangle u_n \, ,
\end{align}
for a suitable function $g_\delta$ that remains bounded for all $\delta>0$ but for which $g_\delta(\sigma) \rightarrow  1/\sigma$ as $\delta \rightarrow 0$. With a suitable speed of such a pointwise convergence, the continuous dependence on the data, i.e., $\reconstructionOperator(\data;g_\delta) \rightarrow \operator^\dagger \dataNoisefree$, can be reestablished. \review{An extensive overview on requirements for this kind of convergence and classical examples for $g_\delta$ including Tikhonov, Lavrentiev or truncated SVD regularization, can be found in \cite{engl1996regularization}.} 

\noindent In the specific case of the linear operator $$\operator:  L_2(\mathbb{R}^2)  \rightarrow L_2(\mathbb{R}\review{\times} [0,\pi])$$ being the Radon operator defined on functions on the whole space (which is not compact in contrast to integral operators on bounded domains), we have to work with a continuous spectrum.
The most commonly used reconstruction technique for inverting the Radon transform is the \textit{filtered backprojection}, which follows a very similar strategy to \eqref{eq:reconstructionSVD}, but exploits the structure of the operator in a different way: With the central slice theorem, the inverse Radon transform applied to some \review{range-element} $\dataNoisefree \in \review{\mathcal{R}}(\operator)$ is given by
\begin{equation}
\label{eq:radonInversion}
    \operator^{-1}\dataNoisefree = \operator^*\left(\ift_{\text{\footnotesize 1-D}} \left(\hat{\rho} \cdot \ft_{\text{\footnotesize 1-D}}\review{\dataNoisefree}\right) \right),
\end{equation}
where $\operator^*$ denotes the adjoint operator of $\operator$ (also called \textit{back-projection operator}),  $\hat{\rho}(r) = \lvert r \rvert$ is called the \textit{ramp-filter}, $\ft_{\text{\footnotesize 1-D}}$ is the one-dimensional Fourier transform with respect to the spatial offset variable of the Radon transform, and $r$ (in the definition of $\hat{\rho}$) is the variable resulting from the Fourier transform. 

In analogy to the regularized version \eqref{eq:reconstructionSVD}, the most common classical way to ensure a stable reconstruction is to replace \eqref{eq:radonInversion} by 
\begin{equation}
\label{eq:radonRegularizedInversion}
    \reconstructionOperator(\data;\rho_\delta) = \operator^*\left(\ift_{\text{\footnotesize 1-D}} \left(\rho_\delta \cdot \ft_{\text{\footnotesize 1-D}}\data \right) \right),
\end{equation}
with the filter $\rho_\delta$ chosen to avoid amplification of high frequency components with large frequency $\lvert r\rvert$. Common choices include  the \textit{Hamming} and \textit{ramp / Ram-Lak} filters. 

In this work, we study \eqref{eq:reconstructionSVD} and \eqref{eq:radonRegularizedInversion} for functions $g_\delta$ (respectively $\rho_\delta$) that can be learned from data. In particular, we show that the shape of the optimal functions $g_\delta$ and $\rho_\delta$ can be characterized in closed form. We prove that the learned approaches result in convergent regularization methods under certain conditions, investigate their behavior under different discretizations, and conduct numerical experiments on CT reconstruction problems.

\section{Related Work}\label{sec:relatedWork}
Regularization methods for linear inverse problems in general and manipulations of the singular values in particular, have long been studied in applied mathematics, c.f. the classical reference \cite{engl1996regularization} or the more recent overview \cite{benning2018modern}. Classical examples of \eqref{eq:reconstructionSVD} include Lavrentiev, Tikhonov, or truncated SVD regularization. Subsequently, a lot of research has focused on non-linear regularization techniques, such as variational methods or (inverse) scale space flows. Even more recently, researchers have focused on \textit{learning} reconstruction schemes through neural networks, which tend to show significantly stronger practical performances, but often lack theoretical guarantees, e.g. being convergent regularizations (independent of their discretization) with error bounds in suitable (problem-specific) metrics. We refer to the two overview papers  \cite{engl1996regularization, benning2018modern} for classical and nonlinear regularization theory and recall some machine-learning specific regularization approaches below. 

The simplest form of benefiting from data-driven approaches are pre- or post-processing networks, i.e., parameterized functions $\net$ that are either applied to the data $\data$ before exploiting a classical reconstruction technique, or to a preliminary reconstruction like \eqref{eq:reconstructionSVD}. Common architectures in the area of image reconstruction problems are simple convolutional neural networks (CNNs) or multiscale approaches such as the celebrated U-Net architecture \cite{FBPConvNet, UNet}. Direct reconstructions (with different types of problem specific information being accounted for) can, for instance, be found in \cite{AUTOMAP,iRadonMap,iCTNet}. Natural extensions of pre- and postprocessing networks use $\operator$ and \review{its adjoint }to switch between the measurement and reconstruction spaces with intermediate learnable operations in structures that are often motivated by classical iterative reconstruction/optimization methods such as gradient descent, forward-back\review{ward} splitting or primal-dual approaches, e.g.  LEARN \cite{LEARN}, FISTA-Net \cite{FISTANet} or the learned primal-dual method  \cite{LearnedPrimalDual}. Yet, without further restrictions, such methods do not allow to prove error estimates or convergence results. 

Coming from the perspective of regularization schemes based on variational methods, many approaches have suggested to \text{learn} the regularizer, starting from (sparsity-based) dictionary learning, e.g. \cite{mairal2008supervised, aharon2006k}, over learning (convex and nonconvex) regularizers motivated by sparsity penalties, e.g. \cite{roth2009fields,chen2014insights,kobler2020total}, to schemes that merely learn descent directions \cite{moeller2019controlling} or operators provably being \review{convergent to a global optimum }(\cite{romano2017little}) or playing the role of (\cite{meinhardt2017learning, rick2017one}) a proximal operator, to reparametrizations of the unknown in unsupervised settings (e.g. the \textit{deep image prior} \cite{ulyanov2018deep}) or in a learned latent space of realistic reconstruction (\cite{bora2017compressed,latorre2019fast}). In a similar fashion, deep equilibrium models \cite{bai2019deep} generalize the optimality conditions arising from learning regularizers (c.f. \cite{riccio2022regularization}) and can, for instance, be combined with convex neural networks \cite{amos2017input} for learning suitable regularizations. Yet, the above approaches are either not viewed in the light of infinite dimensional problems, or are non-convex, such that standard regularization results do not apply, or at least require the ability to compute global minimizers (as in the analysis of \cite{li2020nett} or strict assumptions like the tangential cone condition \cite{aspri2020data}). For a recent overview on data-driven approaches in the context of inverse problems, we refer the reader to \cite{arridge2019solving}. 

Even if properties like convexity of the regularizer allow to deduce convergence and/or error estimates, the complex parametrization of the above methods makes it highly difficult to understand \textit{how} a data-driven approach extracts information from the training data, which quantities matter and what properties a learned regularization has. Therefore, our approach in this paper is to extend our previous work \cite{bauermeister2020learning} by studying the simple linear (singular-value based) reconstruction \eqref{eq:reconstructionSVD} as well as the application-specific learning of a filter in Radon-inversion \eqref{eq:radonRegularizedInversion} in order to provide more insights into properties of learned regularization techniques. \review{While optimizing such spectral filters in a data-driven manner has been studied in the past (c.f. \cite{chung2011designing}), to the best of our knowledge it has not been studied to which extent the obtained reconstruction operators fulfill the characteristics of classical regularization methods.}

While a large number of different data-driven approaches have been tailored to CT-reconstruction (c.f. \cite{REDCNN,CTDIP,CTPhysDIP,DictLearnRecon} for particular examples or \cite{LeuschnerSurvey,ZhangSurvey,WangSurvey} for surveys), they largely follow the above categories in terms of theoretical guarantees and an understanding of the underlying learning process. 

\review{Our framework is close to learning the optimal Tikhonov regularizer, which was extended to the infinite dimensional case and studied with a focus on the generalization to unseen data in \cite{alberti21}. Although our approach is more restricted and requires the knowledge of the forward operator and its singular value expansion, it allows for deriving reasonable assumptions to guarantee convergence in the no-noise-limit.}

\section{Supervised Learning of Spectral Regularizations}\label{sec:regularization}
In this Section we derive a closed-form solution for $g_\delta(\sigma_n)$ in \eqref{eq:reconstructionSVD} when the expectation of the squared norm difference between $\unknown$ and the regularization applied to data $\data$ is minimized. Here $\unknown$ and $\data \review{= \operator\unknown + \noise}$ come from a distribution of training data. We then analyze the corresponding regularization operator and show that the operator is a convergent linear regularization.

\subsection{Optimally Learned Spectral Regularization}\label{sec:optimally-learned-spectral}
In this Section, we study the approach of learning the function $g_\delta$ in the approach of \eqref{eq:reconstructionSVD} from a theoretical perspective. First of all note that due to the assumption of $A$ being compact (and hence having a discrete spectrum), only the evaluations of $g_\delta$ at $\sigma_n$ matter, such that we will focus on the optimal values $g_n:= g_\delta(\sigma_n)$ directly. \review{For the sake of simplicity, we can guarantee the well-definedness of $g(\sigma_n) = g_n$ by assuming the singular values $\sigma_n$ to be strictly decreasing and have multiplicity $\mu(\sigma_n) = 1$.} Let us \review{further} assume that our data formation process \eqref{eq:inverseProblem} arises with noise drawn independently from $\unknown$ from a noise distribution parameterized by the noise level $\delta$ with zero mean. 

For a fixed noise level $\delta$, the most common way to learn parameters (i.e. the $g_n$ in our case) is to minimize the expectation of a suitable loss, e.g. the squared norm, over the training distribution of pairs $(\unknown, \data)$ of the desired ground truth $\unknown$ and the noisy data $\data$. Thus, the ideal learned method is obtained by choosing
$$ \overline{g} = \text{arg}\min_g \mathbb{E}\review{_{\unknown,\noise}}( \Vert \unknown - \reconstructionOperator(\data;g) \Vert^2 ),$$
\review{where $\mathbb{E}_{\unknown,\noise}$ denotes the expected value with respect to the joint distribution of the ground truth data and the noise}.
With the spectral decomposition \eqref{eq:reconstructionSVD} we have 
\begin{align*}
     \Vert \unknown - \reconstructionOperator(\data;g) \Vert^2 &= \review{\Vert u_0 \Vert^2 +}  \sum_n ( (1- \sigma_n g_n)\langle \unknown, u_n \rangle + g_n \langle \noise, v_n \rangle)^2 \\
     &= \review{\Vert u_0 \Vert^2 +} \sum_n ( (1- \sigma_n g_n)^2 \langle \unknown, u_n \rangle^2 + g_n^2 \langle \noise, v_n \rangle^2 \\
       & \quad \review{-}  2 (1- \sigma_n g_n)g_n \langle \unknown, u_n \rangle  \langle \noise, v_n \rangle ),
\end{align*}
\review{where $u_0 \in \mathcal{N}(A)$ is the unique projection of $u$ onto the nullspace of the operator.}
Due to the independence of $\unknown$ and $\noise$ together with the zero mean of the noise we find
$$ \mathbb{E}\review{_{\unknown,\noise}}( \langle \unknown, u_n \rangle  \langle \noise, v_n \rangle )  = 0, $$
hence
$$ \mathbb{E}\review{_{\unknown,\noise}}( \Vert \unknown - \reconstructionOperator(\data;g) \Vert^2 ) =  \review{\mathbb{E}_\unknown\left(\Vert u_0 \Vert^2\right) +} \sum_n ( (1- \sigma_n g_n)^2 \Pi_n  + g_n^2 \Delta_n)$$
with 
$$ \Pi_n := \mathbb{E}\review{_{\unknown}}( \langle \unknown, u_n \rangle^2), \qquad \Delta_n := \mathbb{E}\review{_{\noise}}(\langle \noise, v_n \rangle^2 ),$$
\review{where $\mathbb{E}_{\unknown}$ and $\mathbb{E}_{\noise}$ denote the expected values with respect to the marginal distributions of the ground truth data, or the noise, respectively. Since the whole training only makes sense in the chosen spaces $X$ and $Y$ if the ground truth data $u$ are indeed elements of $X$, we shall assume in the following without further notice that
\begin{align*} \sum_n \Pi_n = \sum_n \mathbb{E}_{\unknown}( \langle \unknown, u_n \rangle^2) = \mathbb{E}_{\unknown}\left( \sum_n \langle \unknown, u_n \rangle^2\right) = \mathbb{E}_{\unknown}(\Vert \unknown \Vert^2) < \infty.\end{align*} }
\noindent The above problem can be minimized for each $g_n$ separately by solving a quadratic minimization problem, \review{where a solution is given by}
\begin{equation}\label{eq:optimalg} \overline{g}_n = \frac{\sigma_n \Pi_n}{\sigma_n^2 \Pi_n + \Delta_n}.\end{equation}
\review{To obtain uniqueness of the solution, we assume that $\Pi_n > 0$ or $\Delta_n > 0$ for all $n \in \mathbb{N}$ throughout this paper.}
\review{\begin{remark}
    The above result can be generalized to the case of singular values not strictly decreasing. Since $\sigma_n$ is converging to zero, each singular value has finite multiplicity and we can rewrite them as a sequence of strictly decreasing singular values $\sigma_n$ with multiplicity $\mu(\sigma_n) \geq 1$ and corresponding singular functions $\{u_{nm}\}_{m=1}^{\mu(\sigma_n)}$ and $\{v_{nm}\}_{m=1}^{\mu(\sigma_n)}$, i.e.,
    \[Au = \sum_n \sigma_n \sum_{m = 1}^{\mu(\sigma_n)} \langle u, u_{nm}\rangle v_{nm} \qquad \text{and} \qquad R(f,g) = \sum_n g_n \sum_{m = 1}^{\mu(\sigma_n)} \langle f, v_{nm}\rangle u_{nm}.\] 
    In this case, the optimal coefficients $g_n$ are again given by \eqref{eq:optimalg}, where we use the generalized variance coefficients \[ \Pi_n := \mathbb{E}_u\left( \sum_{m = 1}^{\mu(\sigma_n)}\langle \unknown, u_{nm} \rangle^2\right), \qquad \Delta_n := \mathbb{E}_{\nu}\left( \sum_{m = 1}^{\mu(\sigma_n)}\langle \noise, v_{nm} \rangle^2 \right).\]
\end{remark}}
\noindent We observe that \review{for $\Pi_n > 0$, Equation} \eqref{eq:optimalg} \review{ can be rewritten as 
\begin{align*}
     \overline{g}_n = \frac{\sigma_n}{\sigma_n^2+ \Delta_n/ \Pi_n}
\end{align*}and therefore} has the same form as the classical Tikhonov regularization
$$ \hat{g}_n = \frac{\sigma_n}{\sigma_n^2 + \alpha} $$
with the commonly fixed regularization parameter $\alpha$ being replaced by a componentwise adaptive and data-driven term $\Delta_n/\Pi_n$. \review{We note that the restricted structure of the regularizer prevents it from approximating the nullspace component of the unknown. This is quantified by the term $\mathbb{E}_\unknown\left(\Vert u_0 \Vert^2\right)$ which is independent of the choice of $g$.} For the generalization error this implies
\begin{align*}
    &\min_g \, \mathbb{E}\review{_{\unknown,\noise}}( \Vert \unknown - \reconstructionOperator(\data;g) \Vert^2 )  \review{- \mathbb{E}_\unknown\left(\Vert u_0 \Vert^2\right)}  \\
       & \quad {} = {} \mathbb{E}\review{_{\unknown,\noise}}( \Vert \unknown - \reconstructionOperator(\data;\overline{g}) \Vert^2 ) \review{- \mathbb{E}_\unknown\left(\Vert u_0 \Vert^2\right)}= \sum_n ( (1- \sigma_n \overline{g}_n)^2 \Pi_n  + \overline{g}_n^2 \Delta_n) \\
   & \quad {} = {} \sum_n \left( \left(1- \frac{\sigma_n^2}{\sigma_n^2 + \Delta_n/\Pi_n} \right)^2 \Pi_n  + \frac{\sigma_n^2}{(\sigma_n^2 + \Delta_n/\Pi_n)^2} \Delta_n \right) \\
     & \quad {} = {}  \sum_n \left( \Pi_n + \frac{\sigma_n^2 \Delta_n + \sigma_n^4 \Pi_n - 2 \Pi_n \sigma_n^2 (\sigma_n^2 + \Delta_n/\Pi_n)}{(\sigma_n^2 + \Delta_n / \Pi_n)^2} \right) \\
     & \quad {} = {}  \sum_n \left( \Pi_n - \frac{\sigma_n^2 \Delta_n + \sigma_n^4 \Pi_n}{(\sigma_n^2 + \Delta_n / \Pi_n)^2} \right) = \sum_n \frac{\Pi_n (\sigma_n^2 + \Delta_n / \Pi_n)^2 - \sigma_n^2 \Delta_n - \sigma_n^4 \Pi_n}{(\sigma_n^2 + \Delta_n / \Pi_n)^2} \\
     & \quad {} = {} \sum_n \frac{\sigma_n^2 \Delta_n + \Delta_n^2 / \Pi_n}{(\sigma_n^2 + \Delta_n / \Pi_n)^2} = \sum_n \frac{\Delta_n}{ \sigma_n^2 + \Delta_n / \Pi_n} = \sum_n \frac{\Delta_n \Pi_n}{ \sigma_n^2 \Pi_n + \Delta_n} \, , 
\end{align*}
and hence, 
\begin{align}
    \min_g \, \mathbb{E}\review{_{\unknown,\noise}}( \Vert \unknown - \reconstructionOperator(\data;g) \Vert^2 ) \review{- \mathbb{E}_\unknown\left(\Vert u_0 \Vert^2\right)} = \sum_n \frac{\Delta_n \Pi_n}{ \sigma_n^2 \Pi_n + \Delta_n} \, . \label{eq:optimal-expected-error}
\end{align}

In general we will need an assumption about the structure of the noise in the data set. First of all we interpret the noise level $\delta$ in a statistical sense as a bound for the noise standard deviation, however formulated such that it includes white noise. This leads to 
\begin{equation}
 \delta^2 = \review{\sup_{n \in \mathbb{N}}} \Delta_n,
\end{equation}
which we will assume throughout the paper without further notice. Note that when studying the zero noise limit $\delta \rightarrow 0$ it is natural to interpret $\Delta_n$ as depending on $\delta$ and in these instances we will use the explicit notation $\Delta_n(\delta)$. We will refer to the case of white noise as $\Delta_n(\delta) = \delta^2$ for all $n\in\mathbb{N}$.

Obviously we expect the clean images to be smoother than the noise, which means that their decay in the respective singular basis is faster. We thus formulate the following assumption, which is automatically true for white noise ($\Delta_n = \delta^2$ for all $n$), since $\Pi_n$ goes to zero: 
\begin{assumption}\label{ass:signalnoise}
For $\delta > 0$, there exists $n_0 \in  \mathbb{N}$ such that the sequence $\Delta_n/\Pi_n$ is well-defined for $n \geq n_0$ and diverging to $+\infty$.
\end{assumption}

For some purposes it will suffice to assume a weaker form, which only has a lower bound instead of divergence to infinity:

\begin{assumption}\label{ass:signalnoiseb}
There exists $c > 0$ and  $n_0 \in  \mathbb{N}$ such that $\Delta_n(\delta) \geq c \, \delta^2 \, \Pi_n$ for every $n \geq n_0$ and $\delta > 0$.
\end{assumption}

With Assumption \ref{ass:signalnoiseb} we can conclude that the operator $\reconstructionOperator(\cdot;\overline{g})$ with $\overline{g}$ given by \eqref{eq:optimalg} is a linear and bounded operator for stricly positive $\delta$.

\begin{lemma}
Let Assumption \ref{ass:signalnoiseb} be satisfied, then $\reconstructionOperator(\cdot;\overline{g})$ with $\overline{g}$ as defined in \eqref{eq:optimalg} is a bounded linear operator for $\delta > 0$. 
\begin{proof}
By construction, the operator $\reconstructionOperator$ is linear in its argument $\data$. \review{ We first note that $\Pi_n = 0$ implies $\overline{g}_n= 0$.} For  \review{$\Pi_n > 0$ and} $n \geq n_0$, the condition $\Delta_n \geq c \, \delta^2 \, \Pi_n$ implies
\begin{align*}
    \overline{g}_n = 
\frac{\sigma_n}{\sigma_n^2 + \Delta_n/\Pi_n} \leq \frac{1}{2 \sqrt{\Delta_n/\Pi_n}} \leq \frac{1}{2 \sqrt{c} \delta} \, ,
\end{align*}
due to $\sigma^2_n + \Delta_n/\Pi_n \geq 2\sigma_n\sqrt{\Delta_n/\Pi_n}$. For $n<n_0$ we have
$$\overline{g}_n \leq \frac{1}{\sigma_{n_0}}.$$
Hence, we conclude 
$$\| \reconstructionOperator(\data, \overline{g}) \| \leq \frac{1}{\min\{ \sigma_{n_0},2 \sqrt{c} \delta \}}\| \data \| \qquad \forall \data,$$ which implies that the operator $\reconstructionOperator(\cdot;\overline{g})$ is a bounded linear operator for $\delta > 0$.
\end{proof}
\end{lemma}

\begin{remark}
In practical scenarios with finite (empirical) data $(u^i, f^i)_{i=1,\hdots, N}$ one typically minimizes the empirical risk 
\begin{equation*}
    \frac{1}{N} \sum_{i=1}^N \|u^i - R(f^i; g) \|^2
\end{equation*}
instead of the expectation, leading to coefficients
\begin{equation}\label{eq:empOptSVD}\overline{g}_n^N = \frac{\sigma_n \Pi_n^N \review{+} \Gamma_n^N}{\sigma_n^2 \Pi_n^N + \Delta_n^N + 2 \sigma_n \Gamma_n^N} \, , \end{equation}
with 
$$  \Pi_n^N := \frac{1}N \sum_{i=1}^N \langle u^i, u_n \rangle^2, \quad 
\Delta_n^N :=\frac{1}N \sum_{i=1}^N \langle \nu^i, v_n \rangle^2, \quad 
\Gamma_n^N :=\frac{1}N \sum_{i=1}^N \langle u^i, u_n \rangle ~\langle \nu^i, v_n \rangle \, .$$
Yet, a convergence analysis of the empirical risk minimization requires further (strong) assumptions to control $\Gamma_n^N$, such that we limit ourselves to the analysis of the ideal case \eqref{eq:optimalg} in this work.
\end{remark}
\subsection{Range conditions}\label{sec:rangecon}

We start our analysis of the learned regularization method with an inspection of the range condition (cf. \cite{benning2018modern}), which for Tikhonov-type regularization methods equals the so-called source condition in classical regularization methods (cf. \cite{engl1996regularization}). More precisely, we ask which elements $u \in X$ can be obtained from the regularization operator $\reconstructionOperator(\cdot,\overline{g})$ for some data $f$, i.e. we characterize the range of $\reconstructionOperator(\cdot,\overline{g})$. Intuitively, one might expect that the range of $R$ includes the set of training samples, or in a probabilistic setting the expected smoothness of elements in the range (with expectation over the noise and training images) should equal the expected smoothness of elements of the training images. However, as we shall see below, the expected smoothness of reconstructions is typically higher. 

We start with a characterization of the range condition\review{, where we again denote the range of an operator by $\mathcal{R}$}. 
\begin{proposition} Let $\overline{g}$ be given by \eqref{eq:optimalg}. Then $u \in \mathcal{R}(\reconstructionOperator(\cdot,\overline{g}))$ if and only if
\begin{equation} \label{eq:rangecondition}
\sum_n \frac{\Delta_n^2}{\Pi_n^2 \sigma_n^2} \langle u , u_n \rangle ^2 < \infty .
\end{equation}
If  Assumption \ref{ass:signalnoiseb} is satisfied we have in particular $u \in \mathcal{R}(\operator^*)$ and in the case of white noise the range condition \eqref{eq:rangecondition} holds if and only if
\begin{equation}  \label{eq:whitenoiserangecondition}
\sum_n \frac{\langle u , u_n \rangle^2}{\Pi_n^2 \sigma_n^2}  < \infty .
\end{equation}
\end{proposition}
\begin{proof}
Given the form of $\overline{g}$ we see that 
$u \in \mathcal{R}(\reconstructionOperator(\cdot,\overline{g}))$
if 
$$ \langle u,u_n \rangle = \overline{g}_n \langle f, v_n \rangle$$
for some $f \in Y$, which is further equivalent to $\frac{1}{\overline{g}_n} \langle u,u_n \rangle  \in \ell^2(\mathbb{N})$. We see that 
$$ \frac{1}{\overline{g}_n} = \sigma_n + \frac{\Delta_n}{\sigma_n\Pi_n}$$
and due to the boundedness of $\sigma_n$, we find $\sigma_n \langle u,u_n \rangle  \in \ell^2(\mathbb{N})$ for any $u \in X$. Hence, the range condition is satisfied if an only if 
$$ \frac{\Delta_n}{\sigma_n\Pi_n} \langle u,u_n \rangle \in \ell^2(\mathbb{N}) $$
holds, which is just \eqref{eq:rangecondition}.
Under Assumption \ref{ass:signalnoiseb} we have
$\frac{\Delta_n}{\sigma_n\Pi_n}\geq \frac{c \, \delta^2}{\sigma_n}$, thus
$$ \sum_n \frac{1}{\sigma_n^2} \langle u, u_n \rangle^2 < \infty.$$
This implies $u =A^*w$ with 
$$ w = \sum_n \frac{1}{\sigma_n} \langle u, u_n\rangle v_n \in Y.$$
For white noise, we have $\Delta_n = \delta^2$, which implies that \eqref{eq:rangecondition} is satisfied if and only if \eqref{eq:whitenoiserangecondition} is satisfied.
\end{proof}

Let us mention that under Assumption \ref{ass:signalnoise} the range condition of the learned regularization method is actually stronger than the source condition $u=A^*w$, since we even have
$$ \sum_n \frac{\Delta_n^2}{\Pi_n^2} \langle w, v_n \rangle^2 < \infty \, ,$$
with the weight $\frac{\Delta_n^2}{\Pi_n^2}$ diverging under Assumption \ref{ass:signalnoise}. This indicates that the learned regularization might be highly smoothing or even oversmoothing, which depends on the roughness of the noise compared to the signal, i.e. the quotient $\frac{\Delta_n}{\Pi_n}$.

\review{Although the range of the reconstruction operator is still dense in the range of $A^*$ if $\Pi_n > 0$ for all $n \in \mathbb{N}$}, the oversmoothing effect of the learned spectral regularization method can be made clear by looking at the expected smoothness of the reconstructions, i.e.\review{,} 
\begin{equation}
\tilde{\Pi}_n  = 
\mathbb{E}_{u,\nu}(\langle \reconstructionOperator(Au+\nu,\overline{g}) , u_n \rangle ^2 ) \, .
\end{equation}
A straight-forward computation yields 
\begin{equation}
\tilde{\Pi}_n  = \frac{\sigma_n^2 \Pi_n}{\sigma_n^2 \Pi_n + \Delta_n} \Pi_n = 
\frac{1}{1 + \frac{\Delta_n}{\sigma_n^2 \Pi_n}} \Pi_n. 
\end{equation}
While we expect (at least on average) a similar smoothness of the reconstructions as for the training images, i.e. $\tilde{\Pi}_n \approx \Pi_n$, we find under Assumption \ref{ass:signalnoiseb} that
$$ \tilde{\Pi}_n  \leq \frac{1}{1 + \frac{c \delta^2}{\sigma_n^2 }} \Pi_n \approx \frac{\sigma_n^2}{c\delta^2} \Pi_n, $$
for large $n$, i.e. $\frac{\tilde \Pi_n}{\Pi_n}$ converges to zero at least as fast as $\sigma_n^2$. The distribution of reconstruction thus has much more mass on smoother elements of $X$ than the training set. The main reason for this behaviour seems to stem from the noise in the data. Thus, although the method is supervised and tries to match all reconstructions to the training data, it needs to produce a linear operator that maps noise to suitable elements of $X$ as well.  Note that the behaviour changes with $\delta \rightarrow 0$. For \review{small $\frac{\Delta_n}{\sigma_n^2 \Pi_n}$}, the denominator is approximately equal to one, hence 
$$ \tilde \Pi_n \approx \Pi_n,$$
i.e. in the limit the right degree of smoothness will be restored.

\review{Another perspective on the expected smoothness is to investigate the bias of the obtained reconstruction operator, or, since we cannot expect to reconstruct nullspace elements \begin{align*}
    \Bar{e}_0 := \mathbb{E}_u\left ( \| u - R(Au,\,\overline{g})\|^2 \right) &- \mathbb{E}_\unknown\left(\Vert u_0 \Vert^2\right).
\end{align*}
Inserting the definition of $\overline{g}$ we derive with that for any $N \in \mathbb{N}$
\begin{align*}
     \Bar{e}_0 = \sum_n \frac{1}{\left(\frac{\sigma_n^2 \Pi_n}{\Delta_n} + 1\right)^2}\Pi_n \leq \frac{\delta^2}{\min_{\substack{n \leq N \\ \Pi_n > 0}} (\sigma_n^2 \Pi_n)}\sum_{n \leq N} \Pi_n + \sum_{n > N} \Pi_n,
\end{align*} 
which yields convergence of the bias to zero as $\delta \rightarrow 0$. We see that, just like the velocity of the convergence of the variance coefficients $\Tilde{\Pi}_n$, the velocity of the convergence of the bias is determined by the convergence of the factors $\frac{\Delta_n}{\sigma_n^2 \Pi_n} \rightarrow 0$.}
\subsection{Convergence of the Regularization}\label{sec:convergence}

In the following we will investigate the convergence properties of the learned spectral regularization as the noise level $\delta$ tends to zero. 
For this sake we will write $\Delta_n(\delta)$ and \review{$\overline{g}(\delta)$} here to make clear that we are looking at a sequence converging to zero. 

\noindent We shall naturally consider convergence in mean-squared-error, i.e. 
\begin{equation}
e(u,\delta) := \mathbb{E}_\nu(\Vert \unknown - \reconstructionOperator(\data;\overline{g}\review{(\delta)})\Vert^2) \label{eq:wasserstein-distance}
\end{equation}
This equals the $2$-Wasserstein distance (cf. \review{\cite[Sec. 2.1]{Ambrosio2013}}) between the concentrated measure at $u$ and the distribution generated by the regularization applied to random noisy data, thus can be understood as a concentration of measure in the zero noise limit.

In addition to the convergence of the regularization for single elements $u \in X$ we can also look at the mean-squared error over the whole distribution of given images and noise, i.e. 
\begin{align} \begin{aligned}\overline{e}(\delta) :=   \mathbb{E}_u(e(u,\delta) ) &\review{- \mathbb{E}_\unknown\left(\Vert u_0 \Vert^2\right)} = \\ &\mathbb{E}_{u,\nu}(\Vert \unknown - \reconstructionOperator(\data;\overline{g}\review{(\delta)})\Vert^2) \review{- \mathbb{E}_\unknown\left(\Vert u_0 \Vert^2\right)} \, \label{eq:mse-whole-distribution}\end{aligned}\end{align}
which is identical to the quantity \eqref{eq:optimal-expected-error} with emphasis of the dependency of $\Delta_n$ on $\delta$.

With the following theorem we verify that $e(u, \delta)$ as defined in \eqref{eq:wasserstein-distance} and $\overline{e}(\delta)$ as defined in \eqref{eq:mse-whole-distribution} are indeed converging to zero for $\delta$ converging to zero.

\begin{theorem} \review{For $u \in \mathcal{N}(A)^\perp$, let $\langle u,u_n\rangle = 0$ for all indices $n$ with $\Pi_n = 0$. Then} $e(u,\delta) \rightarrow 0$ as $\delta \rightarrow 0$. Moreover, we also have $\overline{e}(\delta) \rightarrow 0$ for $\delta \rightarrow 0$.
\end{theorem}
\begin{proof}
We start by computing
\begin{align*}
    \Vert \unknown - \reconstructionOperator(\data; \overline{g}\review{(\delta)}) \Vert^2 = \sum_{n: \Pi_n > 0} &\left[ (1 - \sigma_n \overline{g}_n\review{(\delta)})^2 \langle \unknown, u_n \rangle^2 + \overline{g}\review{(\delta)}^2 \langle \noise, v_n \rangle^2 \right.\\
    &\left. + 2(1 - \sigma_n \overline{g}_n\review{(\delta)})\overline{g}_n\review{(\delta)} \langle \unknown, u_n \rangle \langle \noise, v_n \rangle \vphantom{(1 - \sigma_n \overline{g}_n\review{(\delta)})^2} \right] \,.
\end{align*}
Taking expectations with respect to the noise yields
\begin{align*}
    \mathbb{E}_\nu(\Vert \unknown - \reconstructionOperator(\data; \overline{g}\review{(\delta)}) \Vert^2) &= \sum_{\substack{n, \\ \review{\Pi_n > 0}}} \left[ (1 - \sigma_n \overline{g}_n\review{(\delta)})^2 \langle \unknown, u_n \rangle^2 + \overline{g}_n\review{(\delta)}^2 \Delta_n\review{(\delta)} \right] \\
    &= \sum_{\substack{n, \\\review{\Pi_n > 0}}} \frac{\Delta_n\review{(\delta)}^2 \langle \unknown, u_n \rangle^2  + \sigma_n^2 \Pi_n^2 \Delta_n\review{(\delta)}}{(\sigma_n^2  \Pi_n + \Delta_n\review{(\delta)})^2} \\
    & = \sum_{\substack{n, \\ \review{\Pi_n > 0}}} \left(\frac{\Delta_n\review{(\delta)}}{\sigma_n^2  \Pi_n + \Delta_n\review{(\delta)}}\right)^2 \langle \unknown, u_n \rangle^2 + \frac{\sigma_n^2\Delta_n \review{(\delta)}}{(\sigma_n^2  + \frac{\Delta_n\review{(\delta)}}{\Pi_n})^2} \\
    &\leq \sum_{\substack{n \leq N, \\ \review{\Pi_n > 0}}} \frac{\delta^2}{\sigma_n^2  \Pi_n + \Delta_n\review{(\delta)}}\langle \unknown, u_n \rangle^2   + \sum_{\substack{n > N, \\ \review{\Pi_n > 0}}} \langle \unknown, u_n \rangle^2   +\\
    &~~~ \sum_{\substack{n \leq N, \\ \review{\Pi_n > 0}}}   \frac{\sigma_n^2 \delta^2}{\sigma_n^4} + \sum_{n > N} \Pi_n\\
 &\leq \delta^2 \left(\frac{\Vert u \Vert^2}{\review{\min_{\substack{n \leq N, \\ \Pi_n > 0}} \sigma_n^2 \Pi_n}} + \frac{N}{\sigma_N^2} \right) + 
  \sum_{n>N} (\langle \unknown, u_n \rangle^2 + \Pi_n).
\end{align*}
Now let $\epsilon > 0$ be arbitrary. Due to the summability of $\langle \unknown, u_n \rangle^2 + \Pi_n$ we can choose $N$ large enough such that
$$ \sum_{n>N} (\langle \unknown, u_n \rangle^2 + \Pi_n) < \frac{\epsilon}2.$$
Now we can choose
\review{$$ 
\delta < \sqrt{\frac{\epsilon}{2}} \left(\frac{\Vert u \Vert^2}{\min_{\substack{n \leq N, \\ \Pi_n > 0}} \sigma_n^2 \Pi_n} + \frac{N}{\sigma_N^2} \right)^{-1/2}$$}
to obtain
$$ \mathbb{E}_\nu(\Vert \unknown - \reconstructionOperator(\data; \overline{g}) \Vert^2) < \epsilon$$
for $\delta$ sufficiently small, which implies convergence to zero. 

\noindent In the same way we can estimate
\begin{align*}
    \mathbb{E}_{u,\nu}( \Vert \unknown - \reconstructionOperator(\data;\overline{g}) \Vert^2 ) \review{- \mathbb{E}_\unknown\left(\Vert u_0 \Vert^2\right)} &= \sum_n \frac{\Delta_n \Pi_n}{ \sigma_n^2 \Pi_n + \Delta_n} \, , \\
    &\leq \review{N} \frac{\delta^2}{\sigma_N^2} + \sum_{n > N} \Pi_n \, ,
\end{align*}
where the first equality follows from \eqref{eq:optimal-expected-error}, and obtain convergence by the same argument.  
\end{proof}

\section{Learned Radon Filters}\label{sec:learnedRadon}
In this Section we show the correspondence between data-driven regularization of the inverse Radon transform and filtered back-projection. We therefore analyze the properties of an optimal filter for the regularized filtered backprojection \eqref{eq:radonRegularizedInversion}, i.e., we want to find a function $\overline{\rho}: \mathbb{R} \times [0,\pi] \rightarrow \mathbb{C}$ such that
\begin{equation}\label{eq:learnedbackproj}
    \overline{\rho} = \text{arg}\min_\rho \mathbb{E} \review{_{\unknown, \noise}}( \Vert \unknown - \reconstructionOperator(\data;\rho) \Vert^2 ).
\end{equation}The function $\rho$ now plays the role of the learned spectral coefficients. In particular, $\rho(r, \theta) = \lvert r \rvert$ corresponds to an inversion, i.e., to $g_n = \frac{1}{\sigma_n}$, and $\rho(r, \theta) =  1$ yields the adjoint operator, which corresponds to $g_n = {\sigma_n}$.
In analogy to the derivation of the optimal coefficients in Section \ref{sec:regularization}, we first rewrite the $L^2$-error of the difference of a single $\unknown \in L^2(\mathbb{R}^2)$ and the corresponding regularized reconstruction of the measurement $\data \in L^2(\mathbb{R}^2)$ as follows (using the isometry property of the Fourier transform)
\begin{align*}
    \Vert \unknown - \reconstructionOperator(\data;\rho) \Vert^2 &= \int_{\bR^2} \vert\ft_{\text{\footnotesize 2-D}}\unknown(\xi) - \ft_{\text{\footnotesize 2-D}}\reconstructionOperator(\data;\rho)(\xi)\vert^2 \, \diff \xi\\
    &= \int_0^{\pi} \int_{\bR} \vert r\vert\,\vert\ft_{\text{\footnotesize 2-D}}\unknown(r\Vec{\theta}) - \ft_{\text{\footnotesize 2-D}}\reconstructionOperator(\data;\rho)(r\Vec{\theta})\vert^2 \,\diff r\,\diff \Vec{\theta}\\
    &= \int_0^{\pi} \int_{\bR} \vert r\vert\,\left\vert\ft_{\text{\footnotesize 1-D}}\operator\unknown(r,\theta) - \frac{\rho(r,\theta)}{\vert r \vert}\,\ft_{\text{\footnotesize 1-D}}\data(r,\theta)\right\vert^2 \,\diff r\,\diff \Vec{\theta}.
\end{align*}
In the last line, we use the central slice theorem to get
\begin{equation*}
    \ft_{\text{\footnotesize 2-D}}\unknown(r\Vec{\theta}) = \ft_{\text{\footnotesize 1-D}}\operator\unknown(r,\theta) \qquad \text{and} \qquad \ft_{\text{\footnotesize 2-D}}\reconstructionOperator(\data;\rho)(r\Vec{\theta}) =\frac{\rho(r,\theta)}{\vert r \vert}\,\ft_{\text{\footnotesize 1-D}}\data(r,\theta)
\end{equation*}
where the latter equality holds since the range of the Radon transform is dense in $L^2(\mathbb{R} \times [0,\pi])$ and thus
\begin{equation*}
    A\reconstructionOperator(\data;\rho) = \ift_{\text{\footnotesize 1-D}}\left( \frac{\rho}{\vert \cdot \vert}\, \cdot \ft_{\text{\footnotesize 1-D}}\data\right).
\end{equation*}
Note, that for the reconstruction operator to be well-defined and continuous on $L^2$, the filter function $\rho$ has to be essentially bounded.
With the same arguments used in Section \ref{sec:regularization} we derive
\begin{align*}
   \mathbb{E}( \Vert \unknown - \reconstructionOperator(\data;\rho) \Vert^2 ) = \int_0^{\pi} \int_{\bR} \lvert r\rvert \left(\lvert 1-\psi(r,\theta)\rvert^2 \, \Pi(r, \theta) + \lvert\psi(r,\theta)\rvert^2 \, \Delta(r,\theta)\right)\diff r \diff\theta.
\end{align*}
with $\psi(r,\theta) = \frac{\rho(r,\theta)}{\vert r \vert}$ and
\begin{equation*}
    \Pi(r, \theta) = \mathbb{E}\left( \left\lvert\ft_{\text{\footnotesize 1-D}} Au(r,\theta)\right\rvert^2\right) \qquad \text{and} \qquad \Delta(r, \theta) = \mathbb{E}\left(\left\lvert\ft_{\text{\footnotesize 1-D}} \noise(r,\theta)\right\rvert^2\right).
\end{equation*}
By the above formula we can restrict $\psi$ and thus $\rho$ to be real-valued, as adding an imaginary component would always increase the expected value. Accordingly, the optimal filter is real-valued and given by 
\begin{equation*}
    \overline{\rho}(r,\theta) = \vert r \vert\,\frac{\Pi(r,\theta)}{\Pi(r,\theta) + \Delta(r, \theta)}.
\end{equation*}
To fulfill the boundedness of $\overline{\rho}$ we assume that
\begin{equation*}
    \frac{\Delta(r,\theta)}{\Pi(r,\theta) \lvert r\rvert} \longrightarrow \infty, 
\end{equation*}
 as $\lvert r\rvert \longrightarrow \infty$ for every $\theta \in [0,\pi]$, which can be recognized as the obvious conditions of ground-truth data being smoother than noise on average. The expected error using the optimal filter is then given by 
\begin{equation*}
    \mathbb{E}( \Vert \unknown - \reconstructionOperator(\data;\overline{\rho}) \Vert^2 )= \int_0^{\pi} \int_{\bR} \lvert r\rvert\, \Pi(r, \theta) \left(1- \frac{\Pi(r, \theta)}{\Pi(r, \theta) +  \Delta(r,\theta)}\right) \diff r \diff\theta.
\end{equation*}
Using the central slice theorem again we can further determine the expected smoothness $\Tilde{\Pi}(r,\theta) := \mathbb{E}\left( \left\lvert\ft_{\text{\footnotesize 1-D}}A\reconstructionOperator(\data;\overline{\rho})(r,\theta)\right\rvert^2\right) $ of the reconstructions as
\begin{align*}
    \Tilde{\Pi}(r,\theta) = \left(\frac{\overline{\rho}(r,\theta)}{\vert r \vert}\right)^2\,\mathbb{E}\left( \left\lvert\ft_{\text{\footnotesize 1-D}} \data(r,\theta)\right\rvert^2\right)
    = \frac{\Pi(r,\theta)}{\Pi(r,\theta) + \Delta(r,\theta)} \Pi(r,\theta).
\end{align*} This reveals the analogue to the smoothing effect obtained by the spectral reconstruction that was discussed in Section \ref{sec:rangecon}.

\begin{remark}
Similar to the spectral regularization, optimal filters can also be computed for a finite set of training data. \review{Rewriting the empirical risk for input-output pairs $\{\unknown^i, f^{\delta,i}\}_{i = 1}^N$ as
\begin{align*}
   \frac{1}{N} \sum_{i = 1}^N \Vert \unknown - \reconstructionOperator(f^{\delta,i};\rho) \Vert^2 ) =\int_0^{\pi} \int_{\bR} \lvert r\rvert \left(\lvert 1-\psi(r,\theta)\rvert^2 \, \Pi^N(r, \theta) + \lvert\psi(r,\theta)\rvert^2 \, \Delta^N(r,\theta) \right.&\\\left. - 2(1-\psi(r,\theta))\psi(r,\theta)\Gamma^N \right)\diff r \diff\theta&.
\end{align*}
and computing the optimality conditions for $\psi$ yields}
\begin{equation}\label{eq:empiricalMiniFourier}
    \overline{\psi}^N = \frac{\Pi^N\review{+}\Gamma^N}{\Pi^N+\Delta^N+2\Gamma^N},
\end{equation}
with 
\begin{equation*}
    \Pi^N(r,\theta) = \frac{1}{N} \sum_{i = 1}^N\vert \ft_{\text{\footnotesize 1-D}} A\unknown_i(r,\theta) \vert^2, \qquad  \Delta^N(r,\theta) = \frac{1}{N} \sum_{i = 1}^N\vert \ft_{\text{\footnotesize 1-D}
    } \noise_i(r,\theta) \vert^2
\end{equation*}
and
\begin{equation*}
    \Gamma^N(r,\theta) = \frac{1}{\review{2}N} \sum_{i = 1}^N \left(\ft_{\text{\footnotesize 1-D}} A\unknown_i(r,\theta)\,\overline{\ft_{\text{\footnotesize 1-D}}\noise_i(r, \theta)}  + \overline{\ft_{\text{\footnotesize 1-D}} A\unknown_i(r,\theta)}\,\ft_{\text{\footnotesize 1-D}}\noise_i(r, \theta)\right).
\end{equation*}
\end{remark}

\section{Numerical Experiments with the discretized Radon Transform}\label{sec:numerics}

In this Section we support the theoretical results of the previous Sections with numerical experiments\footnote{Our code is available at\\\url{https://github.com/AlexanderAuras/convergent-data-driven-regularizations-for-ct-reconstruction}.}.

\subsection{Dataset\review{s}}
The following experiments are performed on a \review{synthetic dataset and the LoDoPaB-CT dataset (cf. \cite{Leuschner_2021}). The synthetic dataset contains} $32,000$ images of random ellipses \review{supported on a circle contained completely in the image domain (see Figure \ref{fig:reconstructions} for an example)}. Each ellipse has a random width and height, and is roto-translated by a random distance and angle. All ellipses are completely contained within the image borders, and there may be an overlapping between two or more ellipses.  We split the data into $64\%$ training, $16\%$ validation and $20\%$ test data. \review{The LoDoPaB-CT dataset contains over $40,000$ human chest scans and is split into $35,820$ training, $3,522$ validation, and $3,553$ test images.}
\review{For both datasets, we simulate sinograms as described in the following section and model noisy data by adding Gaussian noise with zero mean and variance $\delta^2$. In Appendix \ref{app:uni} we additionally report results for uniformly distributed noise.} 
\subsection{Discretization and Optimization}
\begin{figure}
    \centering
    \subfloat{\includegraphics[scale = 0.23]{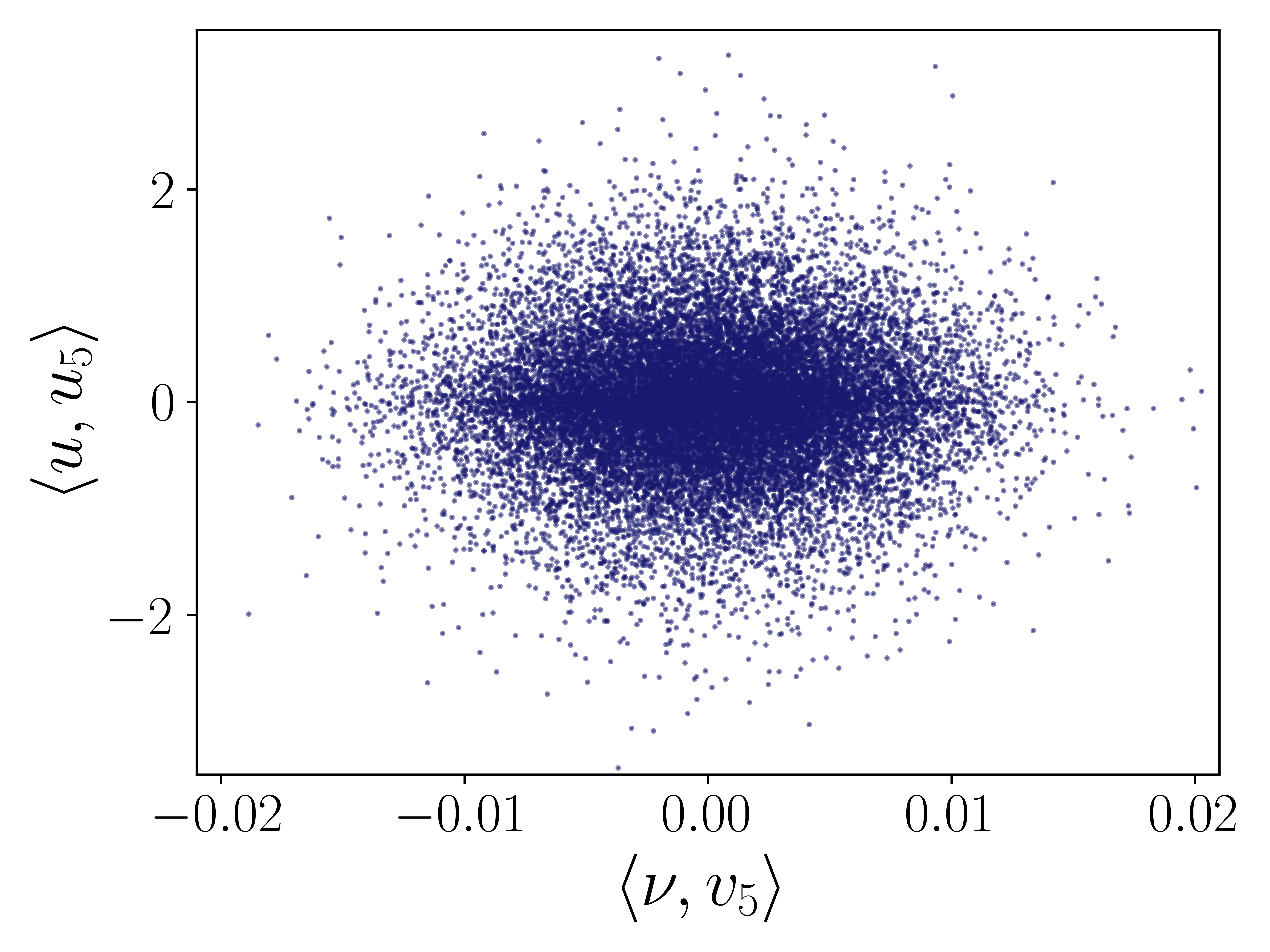}} \,\subfloat{\includegraphics[scale = 0.23]{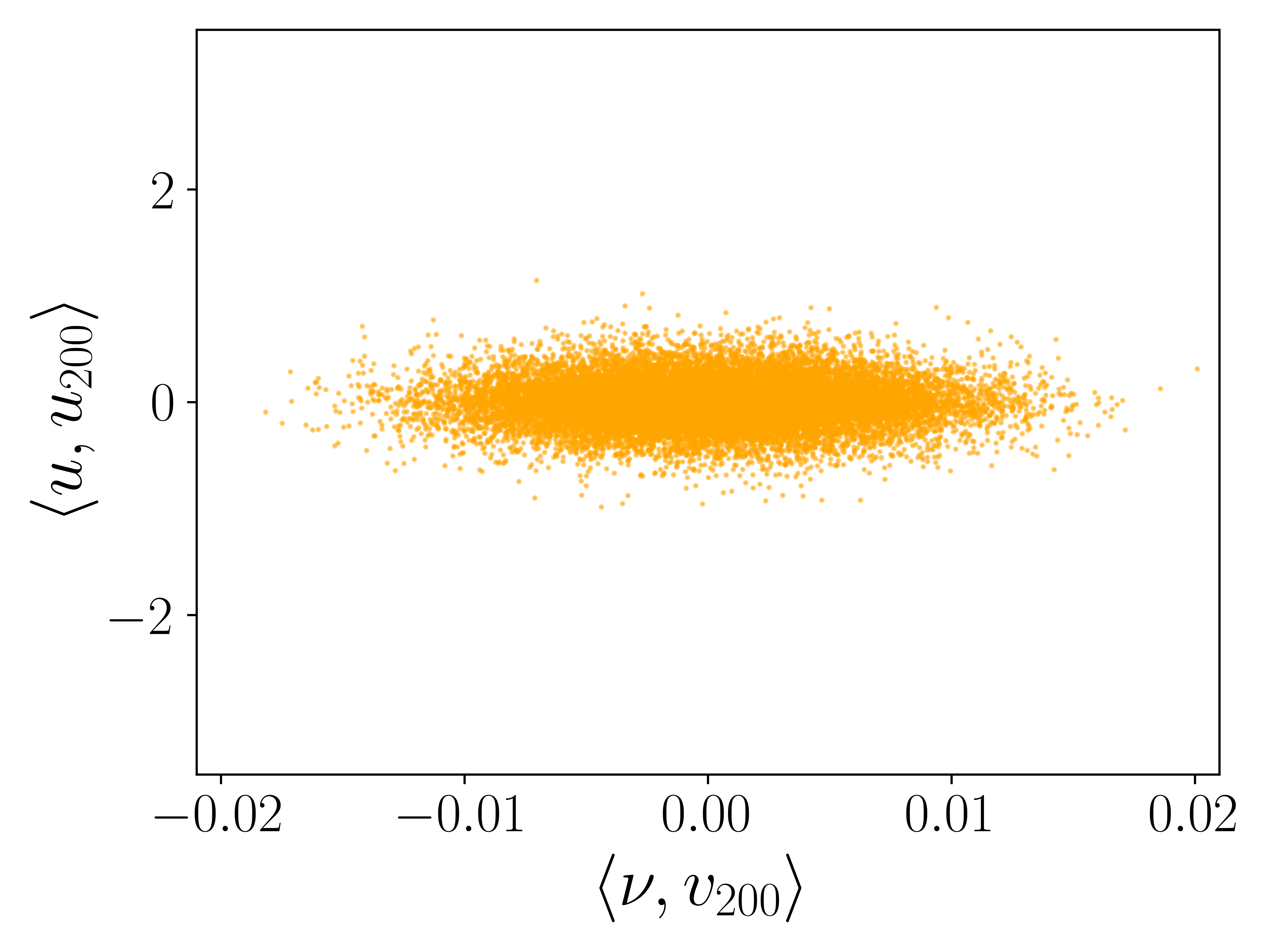}} \,\subfloat{\includegraphics[scale = 0.23]{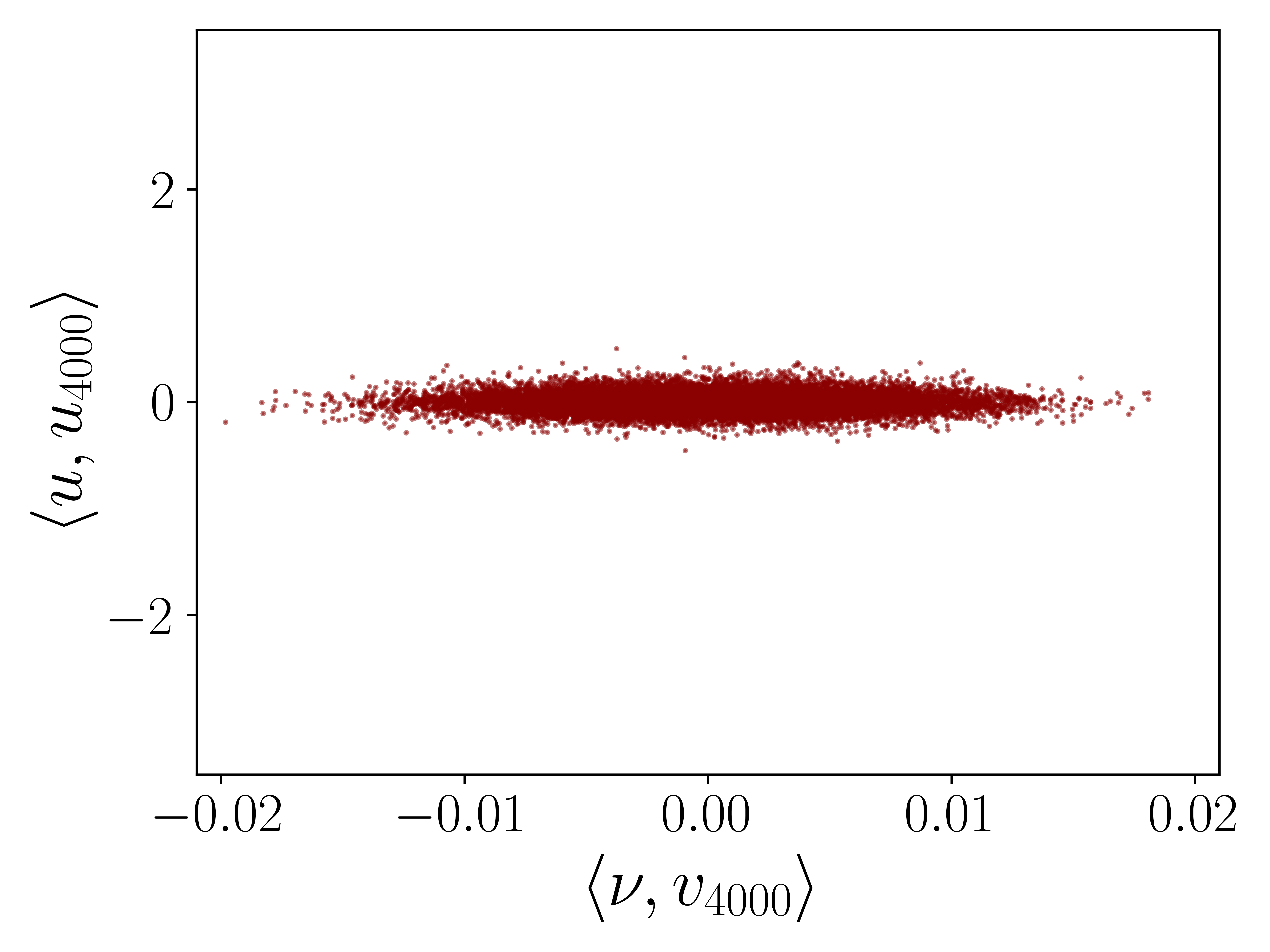}}
    \caption{Comparison of coefficients $\langle u, u_n\rangle$ and $\langle \nu, v_n\rangle$ with $n= 5,\,200,\,4000$ for $64 \times 64$ ground truth images, $93 \times 256$ sinograms and noise level $\review{\delta} = 0.005$. While the variance of the data-coefficients decreases with increasing $n$, the variance of the noise-coefficients is stable.}
    \label{fig:smoothdata}
\end{figure}
We perform numerical experiments with the Radon transform for discrete data \review{based on the Spline-0-Pixel-Model  where we assume the image functions $u\in L^2(\mathbb{R}^2)$ to be of the form 
\begin{align*}
 u(x,y) = \sum_{i,j = 0}^{I-1} u^{(ij)} \,\beta_0\left(I\left(x-\frac{i}{I}\right)\right)\,\beta_0\left(I\left(x-\frac{j}{I}\right)\right) \quad \text{for a.e.} (x,y) \in \mathbb{R}^2,\end{align*}
where $I \in \mathbb{N}$ denotes the number of pixels in each direction, $u_{ij} \in \mathbb{R}$ for each $i,j < I$ and the 0-spline $\beta_0$ is defined as
\begin{align*}
    \beta_0(t) = \begin{cases}
        1 & \text{if } \lvert t \rvert \leq \frac{1}{2},\\
        0 & \text{otherwise.}
    \end{cases}
\end{align*} Accordingly we have $u^{(ij)} = u\left(\frac{i}{I}, \frac{j}{I}\right)$, so with a slight abuse of notation we write $u \in \mathbb{R}^{I\times I}$. Reformulating the Radon transform for this particular choice of input functions yields a finite linear operator (see e.g., \cite{Guedon04} for a detailed derivation)} and can thus be represented by a matrix $A$. For sufficiently small dimensions, the finite set of vectors $u_n$ and $v_n$ is given by the singular value decomposition of $A$, more precisely $A = V\Sigma U^T$, where $u_n$ and $v_n$ are the column vectors of the orthogonal matrices $U$ and $V$ and $\Sigma$ is a diagonal matrix containing the singular values in non-increasing order. Since the considered problem has finite dimensions, Assumption \ref{ass:signalnoiseb} is always fulfilled. We still see a faster decay of $\Pi^N$ compared to $\Delta^N$ as is depicted in Figure \ref{fig:smoothdata}.

Since the computation of the SVD has high memory requirements and the matrix multiplication with $U$ and $V^T$ is computationally expensive, this approach is only feasible for low resolution data. Using the fact that $V$ is orthogonal, we see the connection to the continuous optimization \eqref{eq:learnedbackproj} by rewriting the regularized inversion operator as 
\begin{equation}\label{eq:svdfilter}
        R(\data;\overline{g}^N) = A^T\,\left( V \,\frac{\overline{g}^N}{\text{diag}(\Sigma)} \,V^T\right)\data,
\end{equation} 
where the division is to be interpreted element-wise, and where $\text{diag}(\Sigma)$ returns the main diagonal of $\Sigma$. 

As we will analyze in more detail below, the discretization of the Radon transform leads to \eqref{eq:radonInversion} being incorrect for $\mathcal{F}$ denoting the discrete Fourier transform. Thus, we will compare four different approaches in our numerical results
\begin{enumerate}
    \item The spectral regularization \eqref{eq:reconstructionSVD} with coefficients computed according to the analytical solution \eqref{eq:empOptSVD}.
    \item The spectral regularization \eqref{eq:reconstructionSVD} with coefficients optimized in a standard machine learning setting, i.e., initializing the parameters with zeros, partitioning the data into batches of $32$ instances and then using the PyTorch implementation of Adam \review{(with default parameters except for a learning rate of $0.1$)} for optimization. Ideally, these results should be identical to 1). 
    \item The learned filtered back-projection \eqref{eq:radonRegularizedInversion} with coefficients computed according to the analytical solution \eqref{eq:empiricalMiniFourier}, knowing that the resulting coefficient might not be optimal due to the aforementioned discretization error in \eqref{eq:radonRegularizedInversion}. To partially compensate for this, we replaced the discretization of the ramp filter $\vert r\vert$ by filter coefficients we optimized on clean data. For all non-zero noise levels we kept these pseudo-ramp filter coefficients from the no-noise baseline. 
    \item The learned filtered back-projection \eqref{eq:radonRegularizedInversion} with coefficients that optimize 
    \begin{equation*}
    \overline{\rho}^N = \text{arg}\min_{\rho} \frac{1}N \sum_{i=1}^N \Vert u^i -  A^T\left(F^{-1}  \overline{\rho}^N F\right)f^i\Vert^2,
\end{equation*}
where $F$ denotes the discrete Fourier transform. In keeping with the results from Section \ref{sec:learnedRadon}, we only consider real-valued $\sigma$. The optimization is done in the same way as described above for the SVD. We point out that \review{for rotatation invariant distributions, i.e., cases where the noise and data distributions are constant in the direction of the angles,} the filter $\rho$ can be chosen constant in the direction of the angles \review{as well}. By this the number of necessary parameters is equal to the number of positions $s$. We further note that due to the inevitable discretization errors, we cannot expect perfect reconstructions, not even for perfect measurements without noise. 
\end{enumerate}
We refer to Approaches 2. and 4. as \textit{learned} and to Approaches 1. and 3. as \textit{analytic} and compare their behavior below. Although the discrete Radon transform $A$ is an approximation of the continuous Radon transform (denoted by $\hat{A}$ in the following), several scaling factors have to be taken into account to relate the discrete regularization to a continuous regularization. We collect these factors in the following remark.
\begin{remark}\label{rem:scaling}
     For data supported on the uniform square, we consider discretizations $u \in \bR^{I\times I}$ and $v \in \bR^{K\times L}$ of functions $\hat{u}$ and $\hat{v}$. Here $I$ denotes the number of pixels in each direction, $K$ denotes the number of equispaced angles in $[0,\pi]$ and $L$ denotes the number of equispaced positions in $\left[-\frac{\sqrt{2}}{2},\frac{\sqrt{2}}{2}\right]$. Then, the following approximations hold for
    \begin{itemize}\small\normalfont
        \item[(i)] the Radon transform and its adjoint: \begin{equation*}
            \hat{A}\hat{u} \approx Au, \quad\text{and}\quad \hat{A}^*\hat{v} \approx \frac{\sqrt{2}\pi}{KL} A^T v,
        \end{equation*}
        \item[(ii)] the singular value decomposition:
        \begin{equation*}
            \hat{\sigma}_n \approx  \sqrt{\frac{\sqrt{2}\pi}{KL}}\,\sigma_n,\qquad \hat{u}_n \approx I \,u_n, \quad\text{and}\quad \hat{v}_n \approx \sqrt{\frac{KL}{\sqrt{2}\pi}}\, v_n,
        \end{equation*}
        where $\hat{u}_n$, $\hat{v}_n$ are normalized with respect to the $L^2$-norm, and $u_n$, $v_n$ are normalized with respect to the Euclidean norm,
        \item[(iii)] the data-driven (co-)variance terms:
        \begin{equation*}
            \hat{\Pi}_n = \frac{\Pi_n}{I^2},\qquad \hat{\Delta}_n = \frac{\sqrt{2}\pi}{KL}\,\Delta_n\quad\text{and}\quad \hat{\Gamma}_n =  \sqrt{\frac{\sqrt{2}\pi}{KL}}\,\frac{\Gamma_n}{I}.
        \end{equation*}
    \end{itemize}
\end{remark}
\subsection{Results}
In the following, we present several numerical results for the approaches introduced in the previous Section.
\subsubsection*{Learned vs. analytic}
\begin{table}[]
    \centering
    \begin{tabular}{c!{\vrule width0.7pt}c|c|c|c}
        \textbf{Loss (MSE)} & $\review{\delta} = 0$ & $\review{\delta} = 0.005$ & $\review{\delta} = 0.01$ & $\review{\delta} = 0.015$ \\
        \Xhline{0.7pt}
        SVD analytic & $6.5\cdot10^{-12}$ & $6.7\cdot10^{-4}$& $1.2\cdot10^{-3}$& $1.7\cdot10^{-3}$ \\
        \hline
        SVD learned & $6.6\cdot10^{-6}$ & $6.7 \cdot10^{-4}$ &$1.2\cdot10^{-3}$ & $1.7\cdot10^{-3}$\\
        \hline
        FFT learned & $6.2\cdot10^{-4}$ & $8.7\cdot10^{-4}$ & $1.3 \cdot10^{-3}$ & $1.7\cdot10^{-3}$\\
        \hline
        FFT analytic & $6.1\cdot10^{-4}$ &$1.1\cdot10^{-3}$ & $1.9\cdot10^{-3}$& $2.5\cdot10^{-3}$\\
        \multicolumn{5}{c}{}\\
        \multicolumn{5}{c}{}\\
        \textbf{PSNR} & $\review{\delta} = 0$ & $\review{\delta} = 0.005$ & $\review{\delta} = 0.01$ & $\review{\delta} = 0.015$ \\
         \Xhline{0.7pt}
         SVD analytic& $111.9$ & $31.75$ & $29.09$ & $27.77$ \\
         \hline
         SVD learned & $51.83$ & $31.74$ & $29.94$ & $27.71$ \\
         \hline
         FFT learned & $32.1$ & $30.6$ & $28.83$ & $27.67$ \\
         \hline
         FFT analytic & $32.17$ & $29.65$ & $27.19$ & $26$\\
        \multicolumn{5}{c}{}\\
        \multicolumn{5}{c}{}\\
         \textbf{SSIM} & $\review{\delta} = 0$ & $\review{\delta}= 0.005$ & $\review{\delta} = 0.01$ & $\review{\delta} = 0.015$ \\
         \Xhline{0.7pt}
         SVD analytic & $1$ & $0.832$ & $0.7443$ & $0.6846$\\
         \hline
         SVD learned & $0.999$ & $0.816$ & $0.731$ & $0.683$\\
         \hline
         FFT learned & $0.992$ & $0.837$ & $0.747$ & $0.683$\\
         \hline
         FFT analytic & $0.923$ & $0.901$ & $0.811$ & $0.729$
         
    \end{tabular}

    \caption{Performance of the different approaches on test data for $64 \times 64$ ground truth images and $93 \times 256$ sinograms and different noise levels $\review{\delta}$.}
    \label{tab:results}
\end{table}
\begin{figure}
    \centering
    \subfloat{\includegraphics[scale = 0.35]{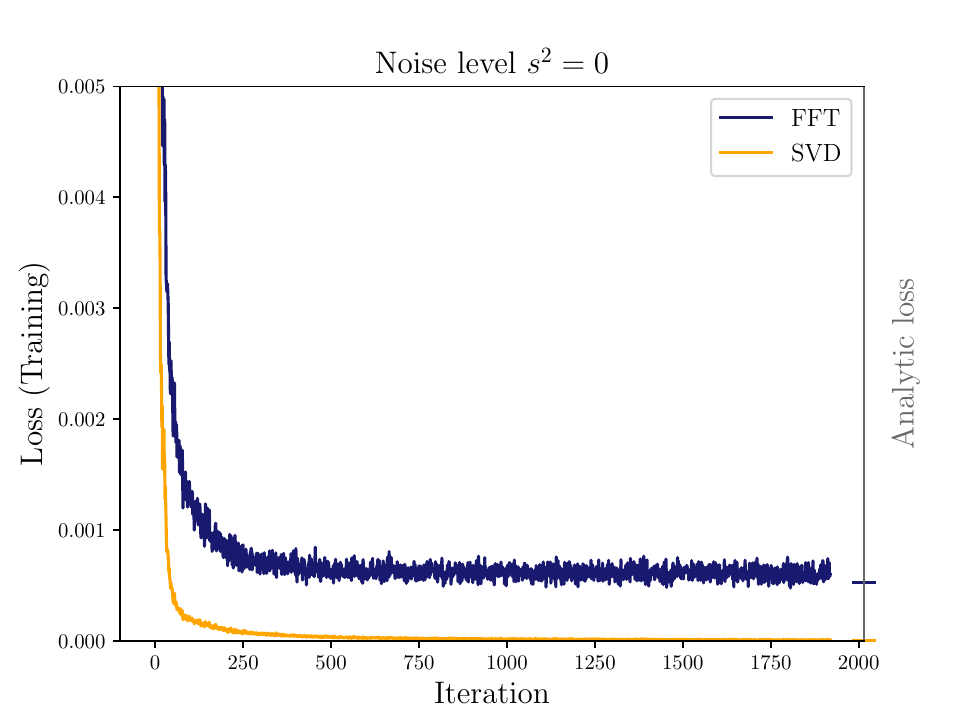}} \subfloat{\includegraphics[scale = 0.35]{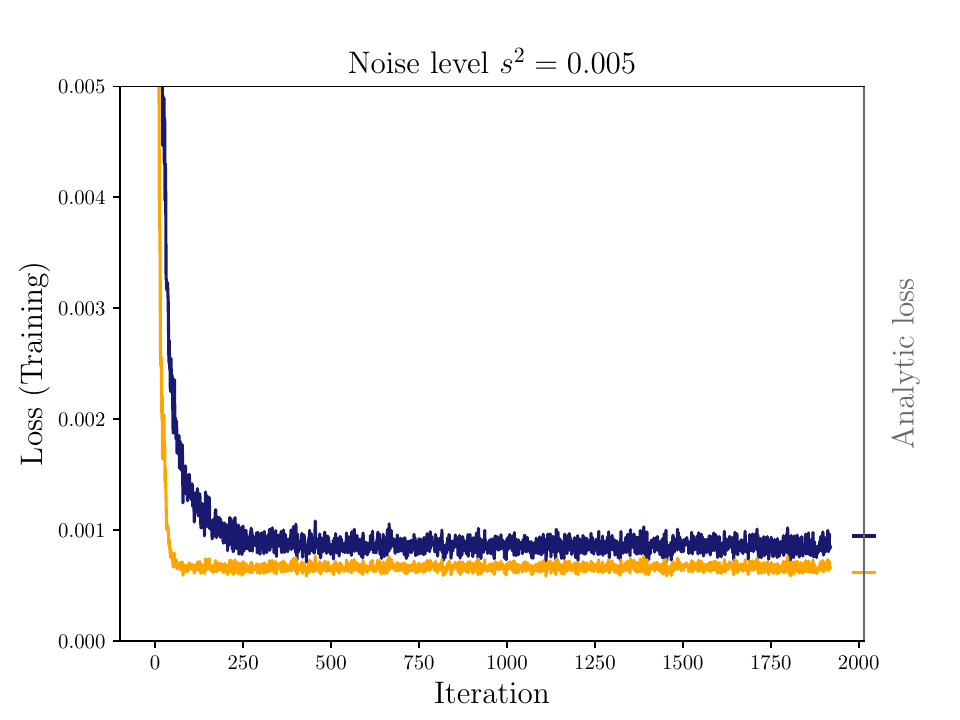}}\\
    \subfloat{\includegraphics[scale = 0.35]{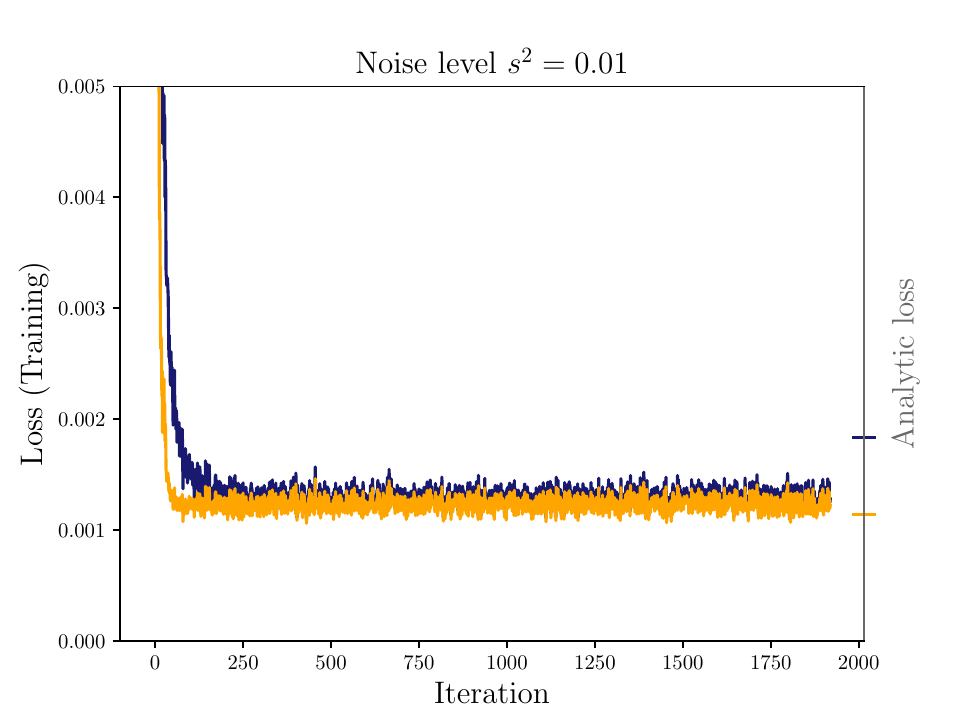}} \subfloat{\includegraphics[scale = 0.35]{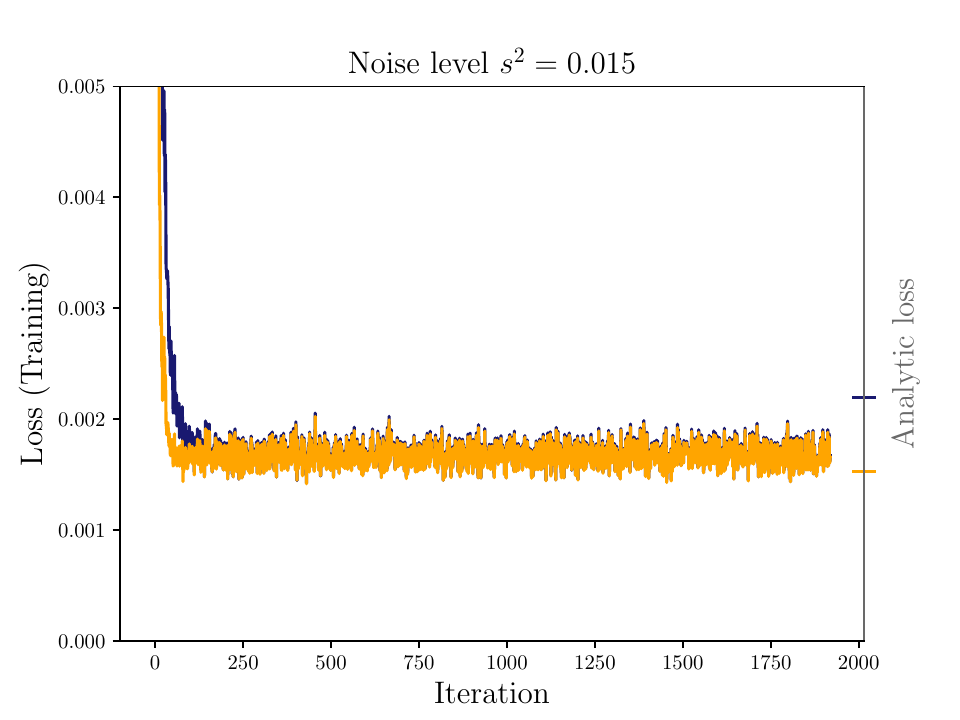}}
    \caption{Training loss curves for both optimization approaches and different noise levels $\review{\delta}$ for $64 \times 64$ ground truth images and $93 \times 256$ sinograms. The scale on the right shows the loss obtained by the analytically computed filters.}
    \label{fig:loss}
\end{figure}
The loss as well as the PSNR- and SSIM-values obtained by the different approaches on the test dataset are documented in Table \ref{tab:results}. Based on these results, we deduce that the SVD-based regularization performs better than the FFT-based regularization. In the SVD case we see that the results obtained by the analytic coefficients are still slightly better than the ones obtained by the learned coefficients. The opposite is true in the FFT case, where the analytic solution performs worse than the learned one. This had to be expected by the discretization errors that are made in the computation of the analytic optimum. Figure \ref{fig:loss} shows the loss curves for both data-driven optimization approaches as well as the loss obtained by the analytically optimal filters. As expected for the SVD-approach the learned loss converges to the analytic loss. On the other hand, the loss approximated with the analytic formula for the Fast-Fourier-approach is higher than the actual data-driven loss, which may again indicate that the optimization balances out the discretization errors. This gap between the continuous and discrete optima can also be seen in higher resolution data, as the comparison of the analytic filters to the learned filters in Figure \ref{fig:filters} shows. Based on these insights, we only show the results for analytic SVD-coefficients and learned FFT-filters in the following.
\begin{figure}
    \centering
    \subfloat{\includegraphics[scale = 0.35]{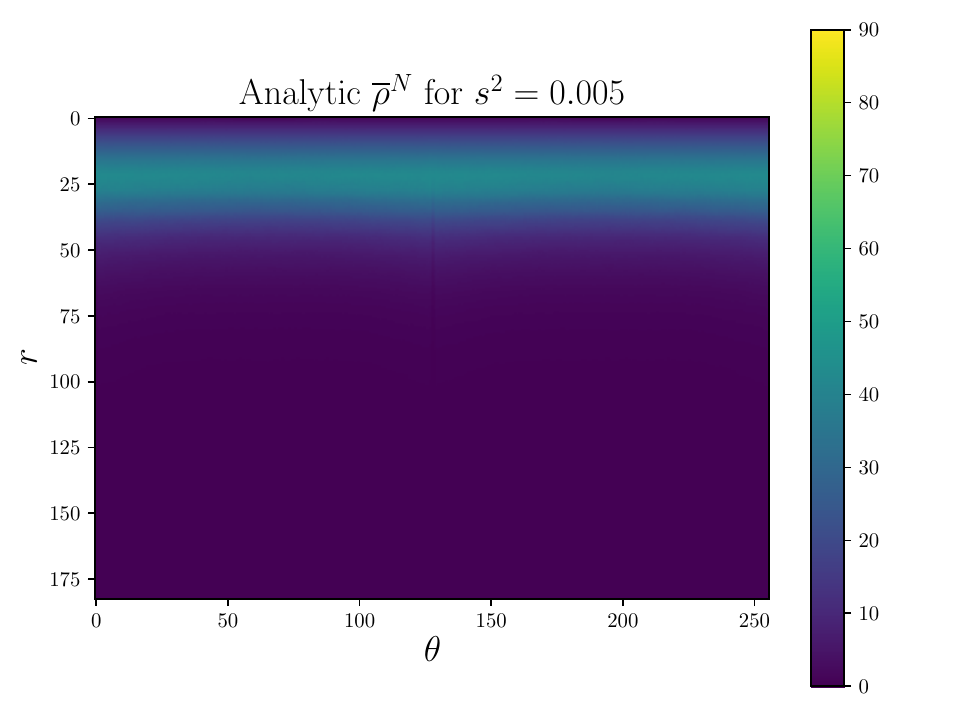}} \subfloat{\includegraphics[scale = 0.35]{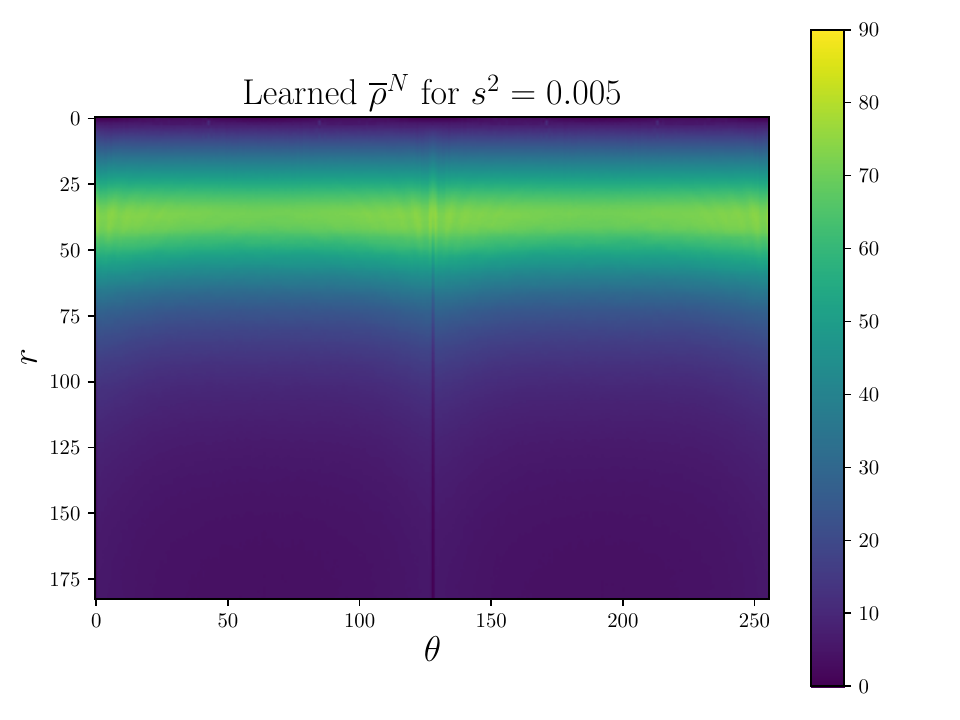}}\\
    \subfloat{\includegraphics[scale = 0.35]{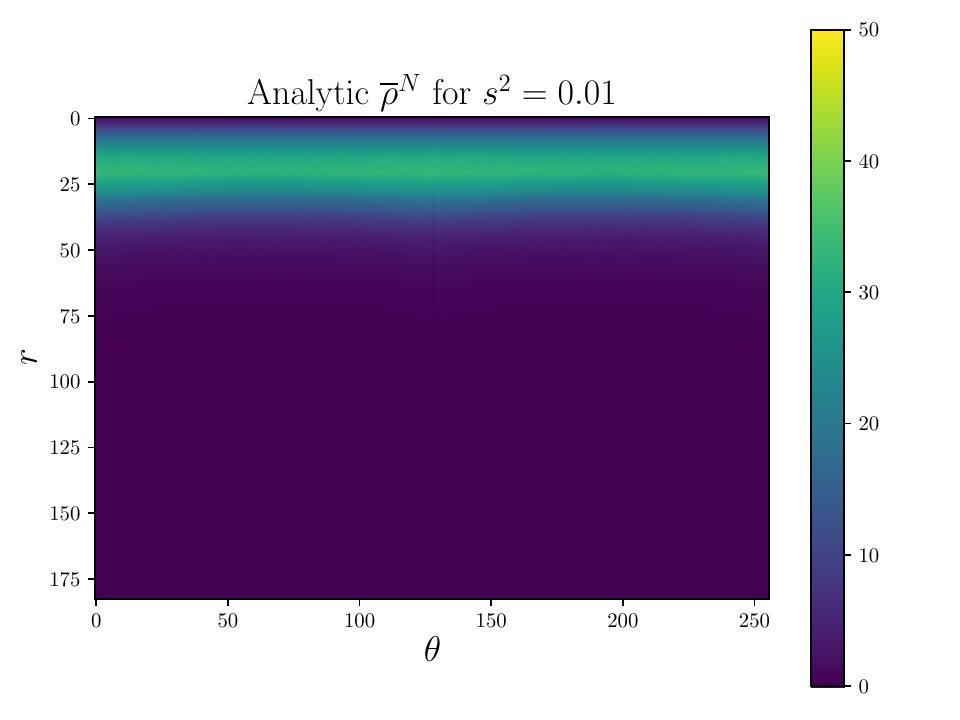}} \subfloat{\includegraphics[scale = 0.35]{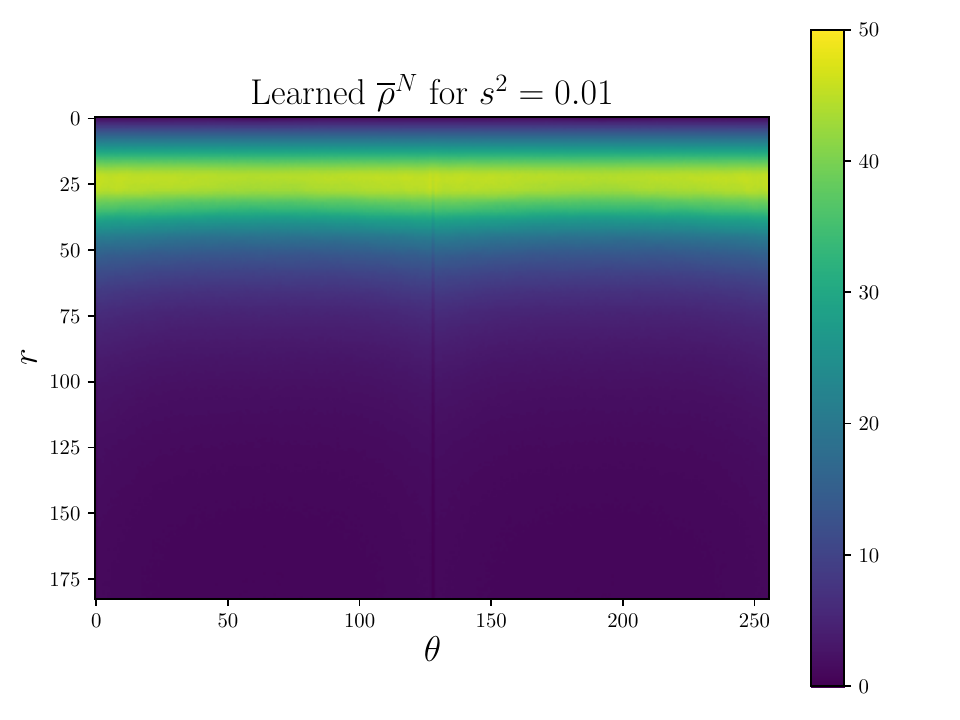}}
    \caption{Comparison of analytic and learned Fast-Fourier-Filters for $256 \times 256$ ground truth images and $365\times 256$ sinograms (values for negative $r$ are omited due to symmetry). }
    \label{fig:filters}
\end{figure}
\begin{figure}
    \centering
    \includegraphics[scale = 0.4]{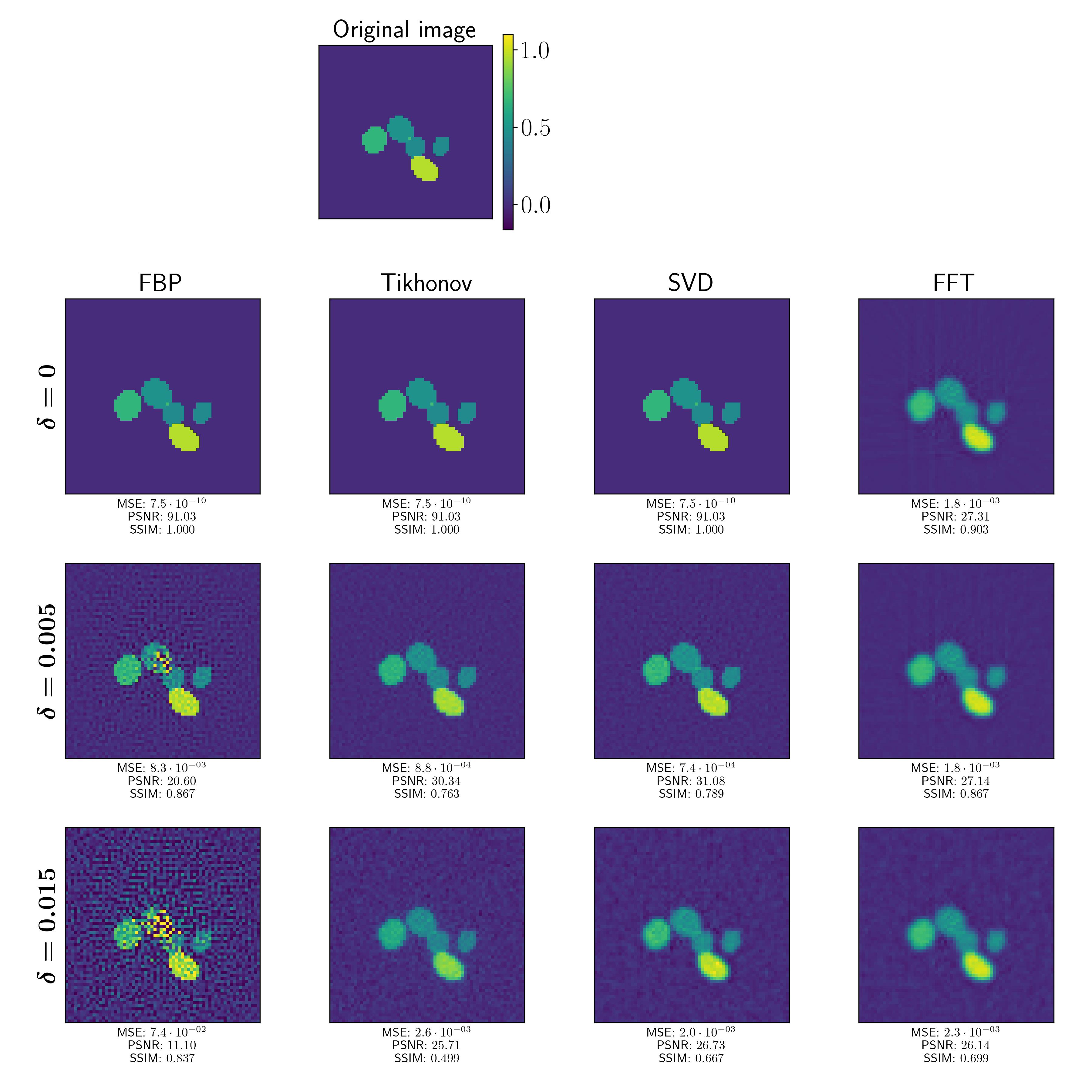}\\
    \caption{Comparison of the ellipse reconstructions obtained by different approaches for different levels of additive gaussian noise. \review{The first row ($\delta = 0$) shows the reconstructed sinograms before adding noise. }}
    \label{fig:reconstructions}
\end{figure}
We further see that the obtained losses of both approaches are more similar the higher the noise level is which is also visible in the reconstructions (see Figure \ref{fig:reconstructions}). \review{Here, we compare our two approaches to na\"ive reconstruction with filtered back-projection (FBP) and Tikhonov reconstruction, where the reconstruction operator is chosen individually for each example with the discrepancy principle.} The reconstruction examples also reveal the predicted oversmoothing effect for both approaches.
\subsubsection*{Generalization to different resolutions}
 Since the operator we use in our experiments is discretized, the optimal coefficients (or filters, respectively) depend on the resolution of the operator and the corresponding data. However, both approaches offer ways to transfer coefficients optimized for a certain resolution to another resolution. One approach to generalize the learned SVD-based coefficients would be to resort to the classic linear regularization formulation \eqref{eq:reconstructionSVD} and retrieve a function $g(\sigma)$ by inter-/extrapolating the optimal values $g(\sigma_n) = \overline{g}_n$. A generalization of the FFT-filters can be obtained in a similar way by interpolation of $\overline{\rho}(\theta, r)$ in the direction of the angles $\theta$ and extrapolation in the direction of the frequencies $r$. Comparing the optimal coefficients obtained for different resolutions after rescaling them in accordance to Remark \ref{rem:scaling} (see Figure \ref{fig:resolutionssvd} for SVD-coefficients and Figure \ref{fig:resolutionsfft} for FFT-filters) we see that all curves behave similarly but there are obvious differences in the magnitude and location of the peak for different resolutions. While the optimal coefficients seem to converge in the domain of small singular values (or low frequencies, respectively), up to the resolutions we were able to compute, there is no obvious convergence visible for the domain of large singular values in the SVD-approach or the domain of high frequencies in the FFT-approach. Note that since high singular values correspond to low frequencies, this reveals an opposite behavior of both approaches. Inspecting the curves for $\hat{\Pi}$ that they are almost identical for high resolutions (see Figure \ref{fig:resolutionssvd}). Therefore, a convergence 
 of the optimal spectral coefficients for higher resolutions seems probable.
\begin{figure}
    \centering
    \subfloat{\includegraphics[scale=0.3]{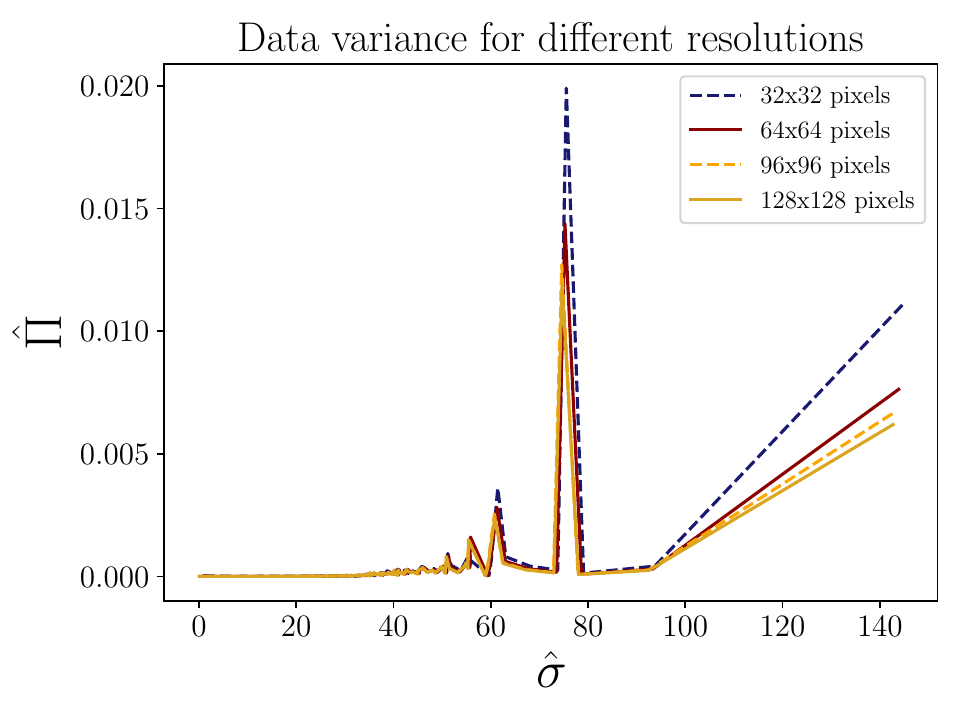}}
    \subfloat{\includegraphics[scale=0.3]{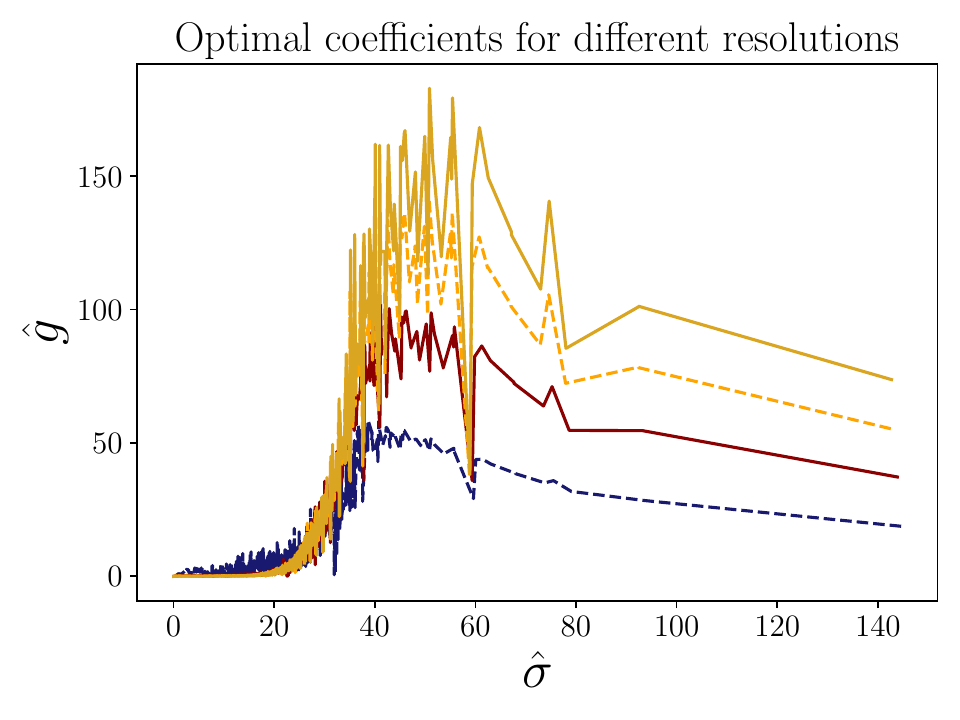}}\\
    \caption{ Comparison of the optimal spectral coefficients for different resolutions of input data noise level $\review{\delta} = 0.005$. All sinograms had a spatial resolution of one ray per pixel and an angular resolution of $256$ angles.}
    \label{fig:resolutionssvd}
\end{figure}
\begin{figure}
    \centering
    \subfloat{\includegraphics[scale = 0.3]{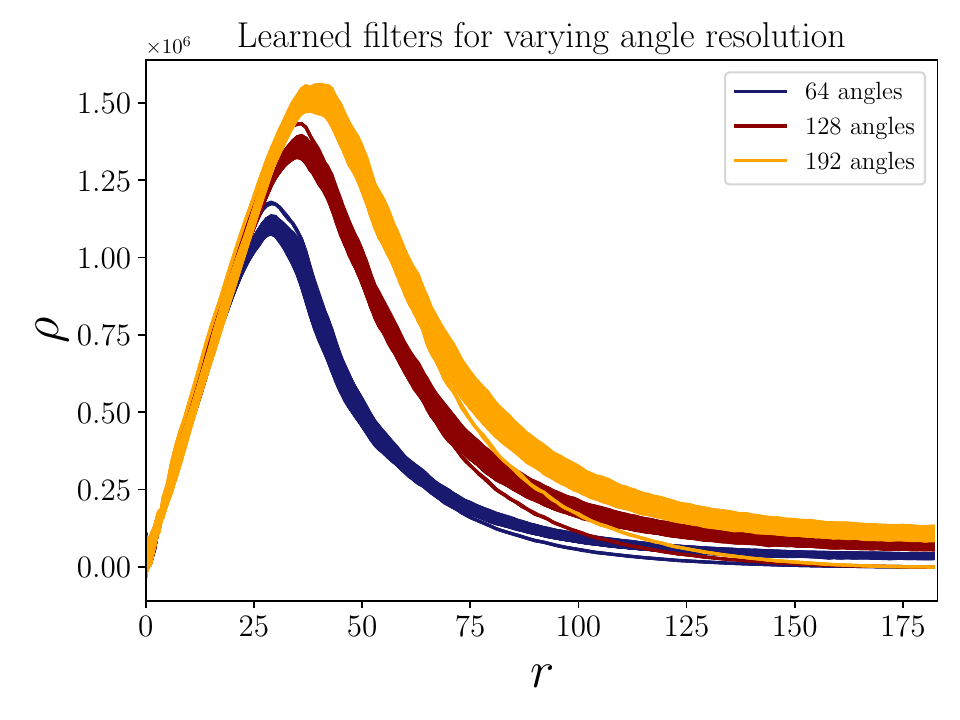}}\; 
    \subfloat{\includegraphics[scale = 0.3]{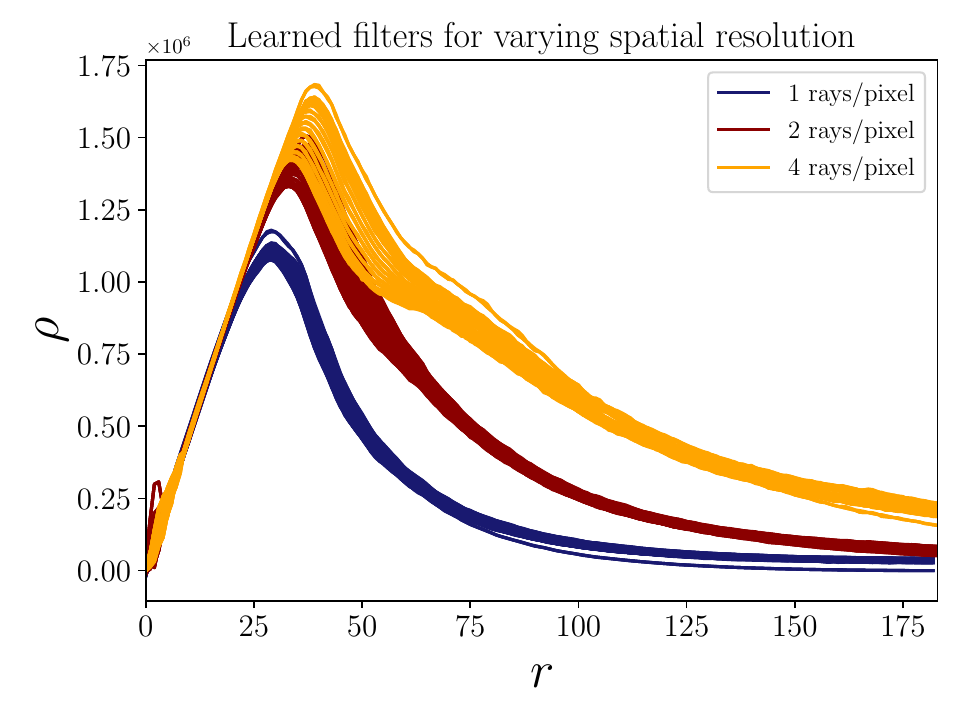}}
     \caption{Comparison of the optimal FFT-filters for different angle and frequency resolutions for $256\times256$ ground truth images with noise level $\review{\delta} = 0.005$. Since the filters are not constant in the direction of the angles, multiple, slightly different curves correspond to one filter. The filters for higher spatial resolutions are cut off at frequency $r = 183$.}
    \label{fig:resolutionsfft}
\end{figure}

\review{\subsection*{Generalization to different data distributions}
\begin{figure}
    \centering
    \includegraphics[scale = 0.4]{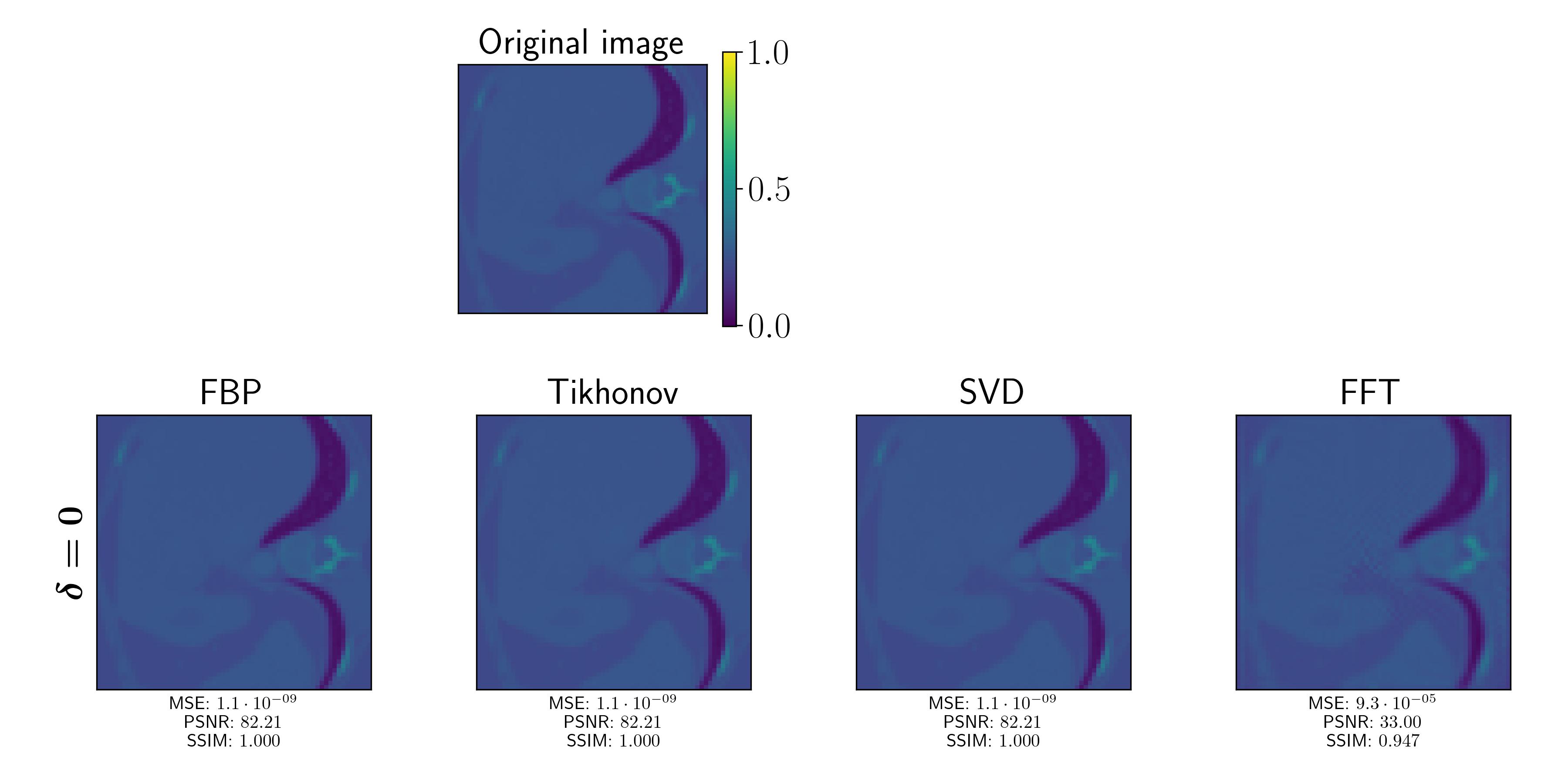}\\
    \caption{\review{Example image from the LoDoPaB-CT dataset and the reconstructions of an uncorrupted sinogram by the different methods.}}
    \label{fig:reconstructions-lodopab}
\end{figure}
\begin{figure}
    \centering
    \includegraphics[scale = 0.4]{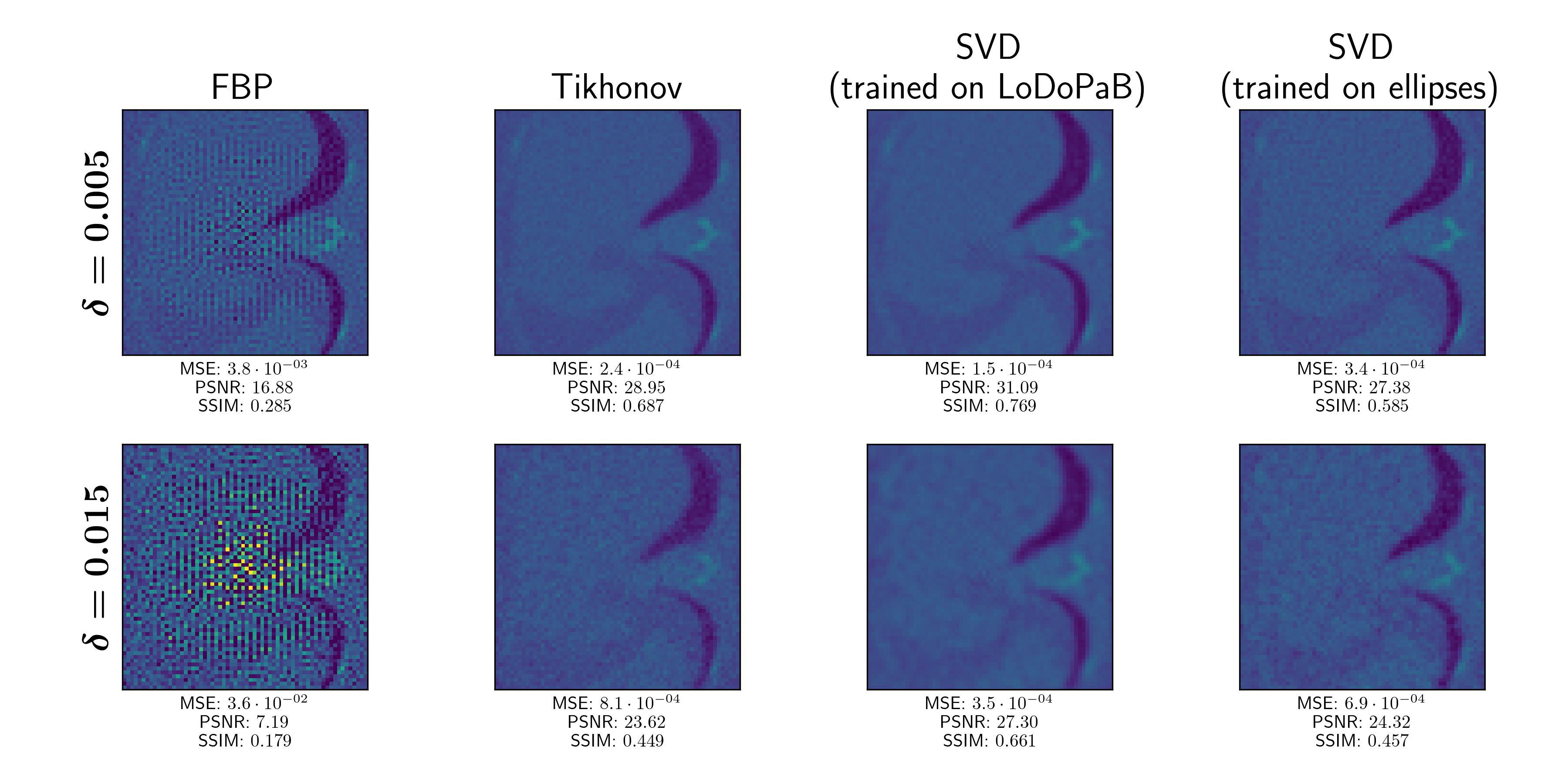}\\
    \caption{\review{Comparison of LoDoPaB reconstructions obtained by SVD-approach, trained on different datasets, for different levels of additive gaussian noise.}}
    \label{fig:reconstructions-trans-svd}
\end{figure}
\begin{figure}
    \centering
    \includegraphics[scale = 0.4]{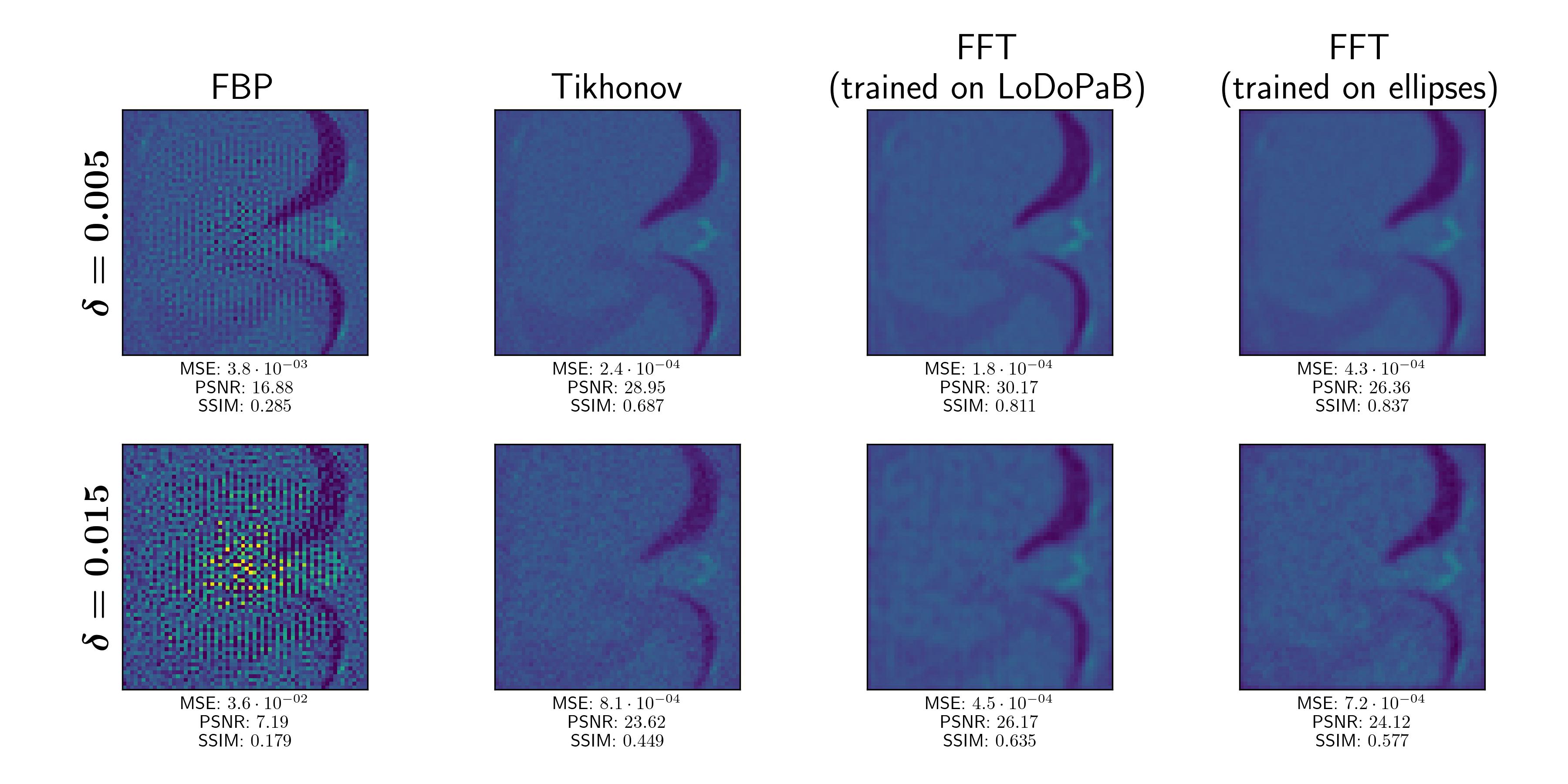}\\
    \caption{\review{Comparison of LoDoPaB reconstructions obtained by FFT-approach, trained on different datasets, for different levels of additive gaussian noise.}}
    \label{fig:reconstructions-trans-fft}
\end{figure}
To evaluate how well the learned regularizers generalize between different datasets, we show empirical results on the LoDoPaB-CT data set in this section. Figures \Cref{fig:reconstructions-lodopab,fig:reconstructions-trans-svd,fig:reconstructions-trans-fft} compare the results obtained with a regularizer trained on LoDoPaB-Ct dataset with the ones obtained with a regularizer trained on the synthetic ellipse dataset. While both approaches clearly improve the reconstructions compared to the filtered back-projection, the effect of optimizing the regularizers on the correct dataset is still clearly visible.}

\section{Conclusions \& Outlook}\label{sec:conclusions}
In this work, we studied the behavior of data-driven linear regularization of inverse problems. For problems emerging from compact linear forward operators we showed that the optimal linear regularization can be computed analytically with the singular value expansion of the operator. An analysis of the range conditions revealed an oversmoothing effect depending on the noise level. We further confirmed that the optimal regularization is convergent under suitable assumptions. By deriving analogous results for the Radon transform, a specific non-compact linear operator, we could establish a connection to the well-known approach of filtered back-projection. Numerical experiments with a discretized (and thus compact) version of the Radon transform verified our theoretical findings for compact operators. We additionally proposed a computationally more efficient discretization of the filtered back-projection by using the Fast Fourier Transform, which allows for finer discretizations of the operator. However, the filters obtained by this approach turned out to strongly deviate from the theoretical results for the continuous case. Especially in the regime of low noise, the discretization errors that come with this approach caused a heavy accuracy loss compared to the SVD-approach. Additionally, we empirically evaluated the convergence behavior of both regularization approaches for the discrete Radon transform with an increasing level of discretization.

For the future, we would like to better understand the discretization error that arises from the discretization of the Radon transform and its impact on the computed filters. It would also be interesting to study more general nonlinear data-driven regularization methods and to find formulations and criteria for which one can prove that these methods are convergent regularizations, similar to the analysis carried out in Section \ref{sec:convergence}.

\section*{Acknowledgements}
\review{A part of this work was carried out while SK and MBu were with the Friedrich-Alexander-Unversität Erlangen-Nürnberg.} SK and MB acknowledge support from DESY (Hamburg, Germany),
a member of the Helmholtz Association HGF. SK, AA, HB, MM and MBu acknowledge the support of the German Research Foundation, projects BU 2327/19-1 and MO 2962/7-1. DR acknowledges support from EPSRC grant EP/R513106/1. MBe acknowledges support from the Alan Turing Institute. This research utilized Queen Mary's Apocrita and Andrena HPC facilities, supported by QMUL Research-IT \url{http://doi.org/10.5281/zenodo.438045}.

\section*{Compliance with Ethical Standards}
\textbf{Conflict of Interest} On behalf of all authors, the corresponding author declares that there is no conflict of interest.

\printbibliography

\appendix
\section{Results for uniformly distributed noise}\label{app:uni}
	In this Appendix, we present the results for the numerical experiments introduced in Section \ref{sec:numerics} with another noise model. More precisely, instead of using centered Gaussian noise with variance $\delta^2$, for a noise level $\delta > 0$, we corrupt the simulated sinograms with noise that is uniformly distributed in the interval $[-\sqrt{3}\delta, \sqrt{3}\delta]$. Based on the reconstructions shown in Figures \ref{fig:reconstructions-elli-uni},  \ref{fig:reconstructions-trans-uni-svd}, and \ref{fig:reconstructions-trans-uni-fft} we deduce that the method is robust with respect to different noise models.
\begin{figure}[h!]
    \centering
    \includegraphics[scale = 0.4]{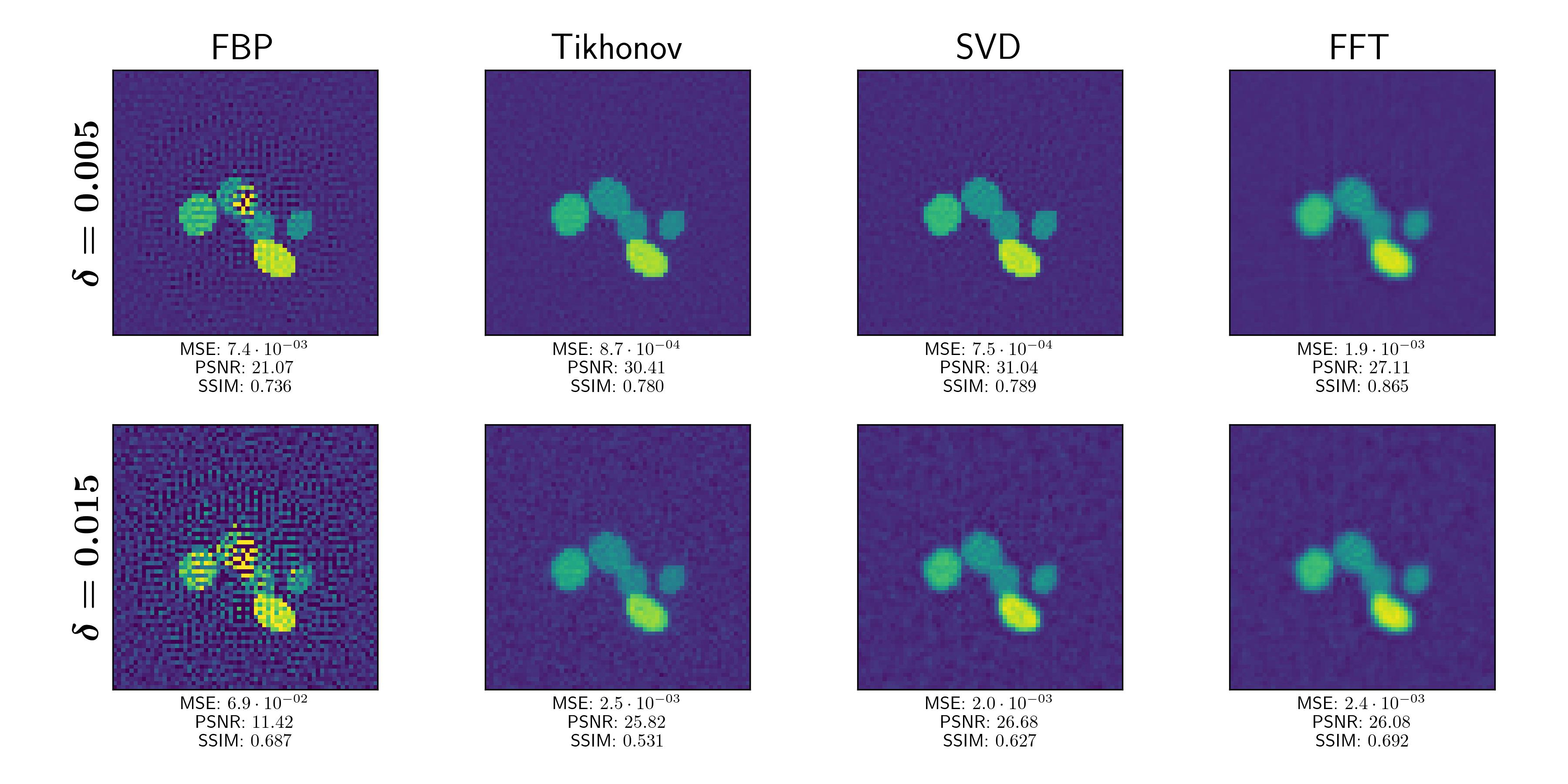}\\
    \caption{Comparison of the ellipse reconstructions obtained by different approaches
for different levels of additive uniform noise.}
    \label{fig:reconstructions-elli-uni}
\end{figure}
\begin{figure}[t!]
    \centering
    \includegraphics[scale = 0.4]{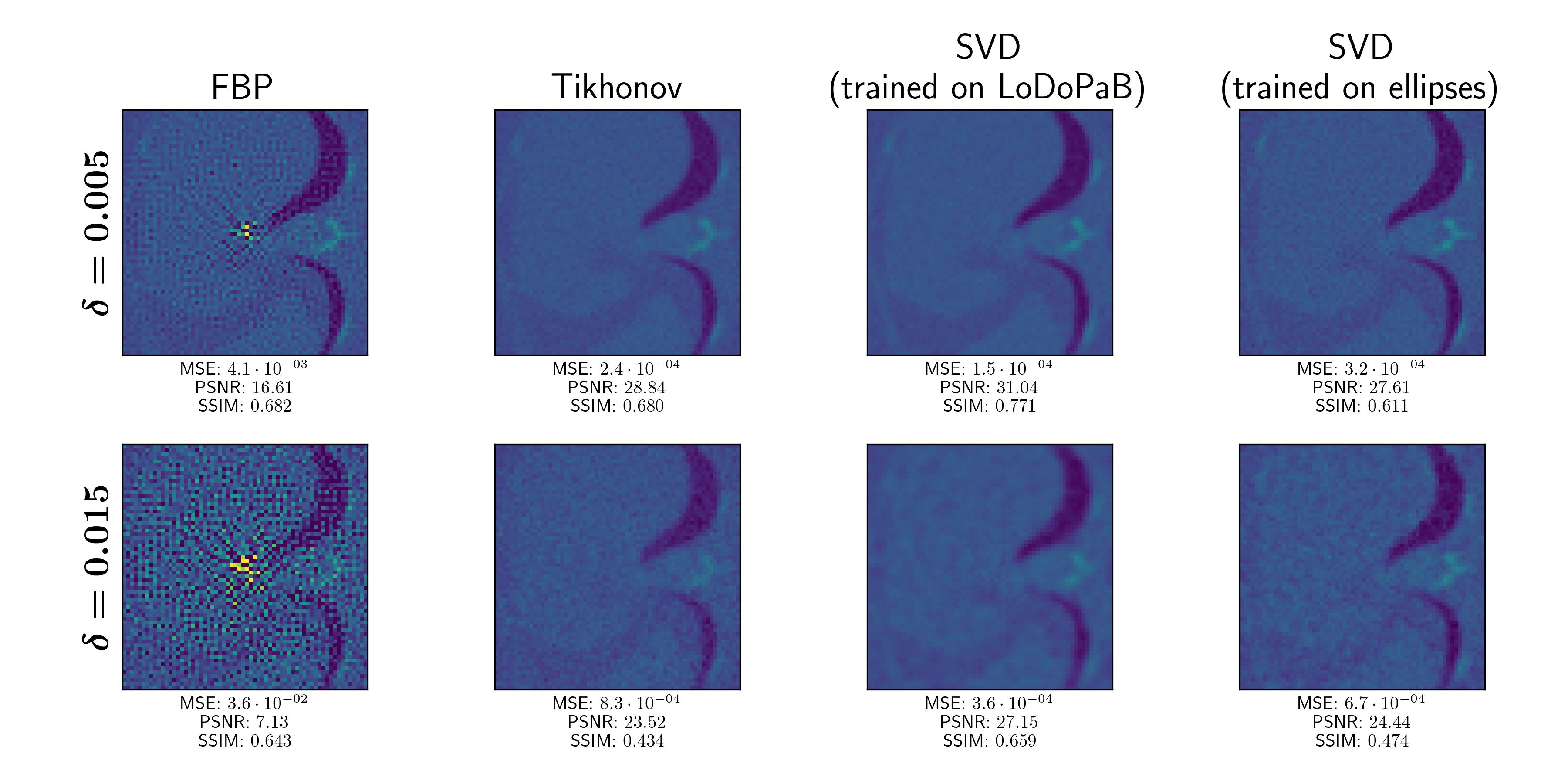}\\
    \caption{Comparison of LoDoPaB reconstructions obtained by SVD-approach, trained on different datasets, for different levels of additive uniform noise.}
    \label{fig:reconstructions-trans-uni-svd}
\end{figure}
\begin{figure}[b!]
    \centering
    \includegraphics[scale = 0.4]{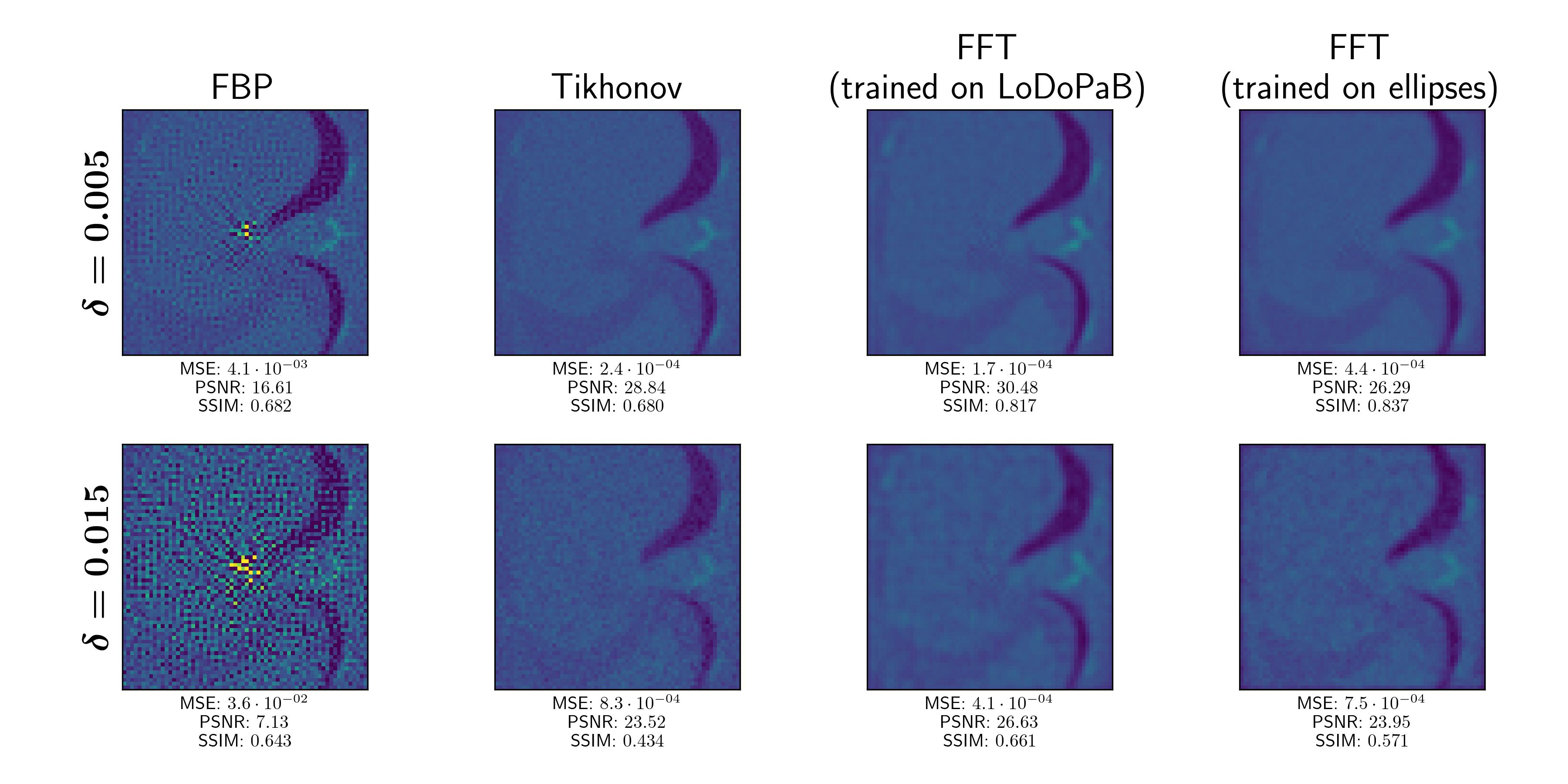}\\
    \caption{Comparison of LoDoPaB reconstructions obtained by SVD-approach, trained on different datasets, for different levels of additive uniform noise.}
    \label{fig:reconstructions-trans-uni-fft}
\end{figure}

\end{document}